\documentclass[11pt,a4paper]{amsart}
\usepackage[utf8]{inputenc}

\usepackage[all]{xy}
\usepackage{latexsym}
\usepackage{hyperref}
\usepackage{breakurl}
\usepackage{mymacro}
\usepackage{tikz-cd}
\usepackage{mathtools}

\usepackage[
backend=biber,
style=ext-numeric,
maxnames = 5,
maxalphanames = 5,
maxcitenames  = 5,
minalphanames = 5,
minnames      = 5,
sorting=nyt,
giveninits,
articlein=false,
url=true,  
doi=false,
eprint=true,
isbn=false
]{biblatex}
\begin{document}

\title{The Rectifiable Rectangular Peg Problem}

\author{Tomohiro Asano}
\address{Department of Mathematics, Kyoto University, Kitashirakawa-Oiwake-Cho, Sakyo-ku, 606-8502, Kyoto, Japan.}
\email{tasano@math.kyoto-u.ac.jp}
\thanks{The first author was supported in part by JSPS KAKENHI Grant Number JP24K16920 and partially by JST, CREST Grant Number JPMJCR24Q1, Japan.}

\author{Yuichi Ike}
\address{Graduate School of Mathematical Sciences, The University of Tokyo, 3-8-1 Komaba Meguro-ku Tokyo 153-8914, Japan.}
\email{ike@ms.u-tokyo.ac.jp}
\thanks{The second author was JSPS KAKENHI Grant Numbers JP21K13801, JP22H05107, and partially by JST, CREST Grant Number JPMJCR24Q1, Japan.}

%    General info

\date{}


\begin{abstract}
    We prove the existence of an inscribed rectangle with any prescribed aspect ratio for a large class of continuous Jordan curves that includes all rectifiable curves and Stromquist's locally monotone curves. 
    Our proof is based on microlocal sheaf theory and is inspired by the recent work of Greene and Lobb.
\end{abstract}

\maketitle

\section{Introduction}

The Square Peg Problem, first posed by Toeplitz~\cite{Toeplitz} in 1911, asks the following: 
\begin{center}
	Does every continuous Jordan curve inscribe a square?
\end{center}
This problem has remained unsolved for over a century. 
In this paper, we consider a more general problem, the so-called Rectangular Peg Problem, which asks whether a Jordan curve inscribes a rectangle with a prescribed aspect ratio. 
See Matschke~\cite{Matschke} for a survey on this topic. 

In \cite{GL21}, Greene and Lobb proved that for every rectangle $R$ in $\bR^2$, every smooth Jordan curve inscribes a rectangle similar to $R$, using some results in symplectic geometry.
More recently, in~\cite{GL24floerhomologysquarepegs}, they showed that for a rectangle $R$ and a rectifiable (i.e., of finite length) Jordan curve satisfying a certain condition, the curve inscribes a rectangle similar to $R$.

In this paper, we give an affirmative answer to the Rectangular Peg Problem for the following class of Jordan curves and any given aspect ratio.
Write $e \colon \bR \to \bR/2\pi \bZ \simeq S^1$ for the quotient map and $\lambda=\xi dx$ for the Liouville 1-form on $\bR^2_{x,\xi}$.

\begin{definition}\label{def:continuous_legendrian_lift}
    A Jordan curve $c \colon S^1 \to \bR^2$ is said to \emph{admit a continuous Legendrian lift} if there exists a sequence of smooth Jordan curves $(c_n \colon S^1 \to \bR^2)_n$ such that 
    \begin{itemize}
        \item[(1)] $(c_n)_n$ converges to $c$ in the $C^0$-sense;
        \item[(2)] defining a function $f_n \colon \bR \to \bR$ by 
        \[
            f_n(t) \coloneqq \int_0^t (c_n \circ e)^*\lambda,
        \]
        there exists a continuous function $f \colon \bR \to \bR$ such that $(f_n)_n$ converges to $f$ uniformly on every compact subset of $\bR$.
    \end{itemize}
\end{definition}

\begin{remark}
    In our subsequent paper~\cite{LiouvilleRegular}, a Jordan curve that admits a continuous Legendrian lift is called \emph{Liouville regular}.
\end{remark}

To describe the aspect ratio of a rectangle, we use the following notion.
For $\theta \in (0,\pi)$, a $\theta$-rectangle is a rectangle in which the angle between the diagonals is $\theta$. 
Note that a $\theta$-rectangle is also a $(\pi-\theta)$-rectangle. 
Our main theorem is stated as follows.

\begin{theorem}\label{thm:main_intro}
    Let $c \colon S^1 \to \bR^2$ be a Jordan curve that admits a continuous Legendrian lift.
    Then $c$ inscribes a $\theta$-rectangle for any $\theta \in (0,\pi)$.
\end{theorem}

We prove that every rectifiable Jordan curve admits a continuous Legendrian lift. 
See \cref{section:curves}. 
As a result, we get:

\begin{corollary}[{\cref{cor:rectifiable_peg}}]\label{cor:intro_rectifiable}
    Every rectifiable Jordan curve inscribes a $\theta$-rectangle for any $\theta \in (0,\pi)$.
\end{corollary}

To the best of our knowledge, this is the first result that gives an affirmative answer to the Square Peg Problem (i.e.,~$\theta=\pi/2$) for all the rectifiable Jordan curves.

There is another large class called locally monotone (see \cref{def:locally_monotone} for the definition).
We prove a locally monotone curve also admits a continuous Legendrian lift, which implies the following:

\begin{corollary}[{\cref{cor:locallymonotone_peg}}]\label{cor:intro_locallymonotone}
    Every locally monotone curve inscribes a $\theta$-rectangle for any $\theta \in (0,\pi)$.
\end{corollary}

\subsection{Our strategy}

We briefly explain our strategy for the proof of the theorem. 
Throughout this paper, for a Jordan curve $c \colon S^1 \to \bR^2$, we write $\cur=c(S^1)$ for its image in $\bR^2$. 
Given a Jordan curve $\cur$, by scaling, we may assume that the area of the open domain bounded by $\cur$ is $\pi$.

The first ingredient is the trick to interpret inscribed $\theta$-rectangles into Lagrangian intersection, which has already appeared in \cite{GL23,GL24floerhomologysquarepegs,Gao2024generic}. 
We identify $\bR^2$ with $\bC$, which we regard as a symplectic manifold. 
Note that if $\cur$ is smooth, it is a Lagrangian submanifold of $\bC$, thus $\cur \times \cur$ is also a Lagrangian submanifold of $\bC \times \bC$.
For $\theta \in [0,\pi]$, define a Hamiltonian diffeomorphism $R_\theta \colon \bC^2 \to \bC^2$ by 
\begin{equation}\label{eq:Rtheta}
    R_\theta =
    \begin{pmatrix} 1 & 1 \\ -1 & 1 \end{pmatrix} ^{-1}
    \begin{pmatrix} 1 & 0 \\ 0 & e^{-\sqrt{-1}\theta} \end{pmatrix}
    \begin{pmatrix} 1 & 1 \\ -1 & 1 \end{pmatrix} .
\end{equation}
One can easily find that a $\theta$-rectangle corresponds to four distinct points $z,w,z',w'$ such that $R_\theta(z',w')=(z,w)$. 
Since $R_\theta(z,z)=(z,z)$, $R_\theta$ is the identity on the diagonal $\Delta_{\bC}$ of $\bC \times \bC$, which corresponds to  degenerate rectangles. 
Thus, the problem of finding a $\theta$-rectangle inscribed in $\cur$ is reduced to finding a intersection point between $\cur \times \cur$ and $R_\theta(\cur \times \cur)$ outside the diagonal $\Delta_C$ of $\cur \times \cur$.

The second ingredient is the following method coming from microlocal sheaf theory, in particular, sheaf quantization. 
For a smooth Jordan curve $\cur$, one can construct a canonical object $F_{\cur}$ in the Tamarkin category whose microsupport is $\cur \times \cur$, called the sheaf quantization of $\cur \times \cur$.
See \cref{section:preliminaries,section:SQ} for more precise definitions. 
By the completeness of the Tamarkin category with respect to the interleaving distance~\cite{AI24completeness,GV24coisotropic}, for a continuous Jordan curve $\cur$, we can still construct its sheaf quantization $F_{\cur}$.
Moreover, by a result in Guillermou--Kashiwara--Schapira~\cite{GKS}, the action $R_\theta$ lifts to the Tamarkin category category.
The Hom space $\Hom(F_{\cur}, R_\theta F_{\cur})$ captures the information of the intersection $(\cur \times \cur) \cap R_\theta(\cur \times \cur)$ and is equipped with a filtration, which can be regarded as a persistence module with structure maps $(\tau_{a,a'})_{a \le a'}$. 
We focus on a ``critical value" $a_0 \in \bR$ such that $\tau_{a,a'}$ is not an isomorphism if $a < a_0 < a'$. 
By the assumption that $c$ admits a continuous Legendrian lift, the diagonal $\Delta_\cur$ contributes only to critical values in $\pi \bZ$. 
We will show that there exists a critical value that is not in $\pi \bZ$, which proves the theorem.
In fact, we give a sheaf-theoretic condition for the existence of a $\theta$-rectangle in \cref{section:sheaf_condition}. 
The conditions in \cref{def:continuous_legendrian_lift} implies that sheaf-theoretic condition. 

The advantages of our sheaf-theoretic approach compared to Floer-theoretic methods are twofold.
First, we can work directly with a continuous Jordan curve $\cur$ by using the sheaf quantization $F_{\cur}$.
This is in contrast to Floer-theoretic methods, which require taking a sequence of smooth objects approximating $\cur$. 
Second, we can use the theory of $\muhom$ to analyze the intersection $(\cur \times \cur) \cap R_\theta(\cur \times \cur)$. 
While $\muhom(F_{\cur},R_\theta F_{\cur})$ corresponds to local Floer cohomology for smooth curves, it cannot be described in terms of a limit or a colimit for continuous curves, since $\muhom$ does not commute with limits or colimits.
Hence, the analysis of the intersection may be easier with the sheaf-theoretic method. 
\smallskip

This paper is organized as follows.
In \cref{section:preliminaries}, we define a twisted version of the Tamarkin category. 
In \cref{section:SQ}, we construct a sheaf quantization of the standard torus and observe some basic properties. 
In \cref{section:sheaf_condition}, we give a sheaf-theoretic condition for the existence of a $\theta$-rectangle.  
In \cref{section:curves}, we prove \cref{thm:main_intro} and show \cref{cor:intro_rectifiable,cor:intro_locallymonotone}.

\subsection{Related work}

We review some history on the Square and Rectangular Peg Problems, restricting our focus to the result relevant to our work.
See Matschke~\cite{Matschke} for a detailed and overall history of this topic. 

Vaughan (published in \cite{Meyerson}) showed that every continuous Jordan curve inscribes a rectangle with a simple topological argument, in which a rectangle on the Jordan curve is interpreted to a immersed point of a surface in a $3$-dimensional space. 
Hugelmeyer~\cite{Hugelmeyer18} proved that for any $n \in \bZ_{\ge 3}$, every smooth Jordan curve has an inscribed rectangle of ratio $\pi k/n$ for some $k \in \{1,\dots,n-1\}$. 
Moreover, he~\cite{Hugelmeyer21} proved that for any smooth Jordan curve, the set of values $\theta \in [0,\pi/2]$ for which the curve inscribe a rectangle of aspect angle $\theta$ has Lebesgue measure at least $\pi/6$. 
In his works, the existence of rectangular pegs is reduced to the existence of intersections of surfaces within a four-dimensional space.
Greene and Lobb~\cite{GL21} solved the Rectangular Peg Problem for smooth Jordan curves using symplectic geometry.  
Moreover, they proved cyclic quadrilateral pegs for smooth curves in~\cite{GL23}.
In \cite{GL24floerhomologysquarepegs}, Greene and Lobb used a version of Lagrangian intersection Floer theory and spectral invariants to prove assertions for rectifiable curves with an additional condition.
Our result is on the line of these.

Our results are also a generalization of the following. 
Emch~\cite{Emch16} proved the existence of an inscribed square for piecewise analytic curves satisfying some additional assumptions. 
Schnirelman~\cite{Schnirelman44} proved it for a class of curves that contains $C^2$, and Stromquist~\cite{Stromquist} proved for locally monotone curves.
Tao~\cite{TAO2017} proved the existence of an inscribed square for a curve that is the union of the graphs of two Lipschitz continuous functions with Lipschitz constant less than~$1$.
Greene--Lobb~\cite{GL24squarepegsgraphs} strengthened Tao's result to the case where the Lipschitz constant is less than $1+\sqrt{2}$. 
Feller--Golla~\cite{FellerGolla} has weakened the regularity condition of the result by Hugelmeyer~\cite{Hugelmeyer18}.

There are also some recent  results~\cite{Gao2024generic,Hugelmeyer2024periodic,GL24polynomial} for related problems with the use of Lagrangian intersection Floer theory.

\section{Preliminaries}\label{section:preliminaries}

Throughout this paper, we set the base field $\bfk$ to be $\bF_2=\bZ/2\bZ$.
Let $X$ be a manifold. 
Let $\pi \colon T^*X \to X$ denote the cotangent bundle and $(x;\xi)$ denote the homogeneous symplectic local coordinate on $T^*X$.
We denote by $\lambda_X = \sum_{i} \xi_i dx_i$ the Liouville $1$-form on $T^*X$. 
We often simply write $\lambda$ for $\lambda_X$.

\begin{notation}
    For objects $F,G$ in a $\bfk$-linear stable ($\infty$-)category, $\Hom (F,G)$ (resp.\ $\End (F)$) denotes the Hom (resp.\ End) object in $\Module (\bfk)$, the presentable stable category of $\bfk$-vector spaces.
    For $v \in H^n(\Hom (F,G))$ (resp.\ $v\in H^n(\End (F))$) for some $n \in \bZ$, we simply write $v\in \Hom (F,G)[n]$ (resp.\ $v\in \End(F)[n]$).
\end{notation}

\subsection{Twisted sheaves}

Let $\Sh(X)$ be the $\bfk$-linear presentable stable category of sheaves of $\bfk$-vector spaces on $X$. 
For each object $F \in \Sh(X)$, we write $\CMS(F)$ for the \emph{conic microsupport}\footnote{In the literature, this is usually called the microsupport, but we use this name for the non-conic microsupport defined below.} of $F$, which is a closed conic subset of $T^*X$. 
For a closed subset $A$ of $T^*X$, we denote by $\Sh_A(X)$ the subcategory of $\Sh(X)$ consisting of objects with conic microsupport contained in $A$. 

In this paper, we use the notion of twisted sheaves. 
We give a short review for twisted sheaves from \cite{KashiwaraRepDflag}. 
Guillermou~\cite{Gu12,Gui23} and Jin~\cite{Jin2020J-homomorphism} used twisted sheaves in the process of constructing sheaf quantizations of compact exact Lagrangian submanifolds in cotangent bundles, 
and we use them in a parallel manner in this work. 
The formulation within the context of $\infty$-categories has been done in \cite{CKNS24}, and we follow their approach. 
See \cite{CKNS24} for the precise definition and treatment of twisted sheaves. 
Here, we only treat very restrictive twistings and one can describe twisted sheaves via untwisted sheaves. See \cref{rem:twisted_sheaves} below. 

Let $\Pic (\bfk)$ be the ($\infty$-)group consisting of the invertible objects in $\Module(\bfk)$. 
In our setting $\bfk=\bZ/2\bZ$, $\Pic (\bfk)$ is isomorphic to $\bZ$ (the element $\bfk[n]\in \Pic(\bfk)$ corresponds to $n\in \bZ$).
Let $\eta \colon X \to B\Pic(\bfk)$ be a twisting. 
We denote $\Sh^\eta(X)$ the category of sheaves on $X$ twisted by $\eta$. 
A homotopy between two twistings $\eta_1$ and $\eta_2$ gives an identification $\Sh^{\eta_1}(X)\simeq \Sh^{\eta_2}(X)$. 
In particular, a null homotopy (to the basepoint) of a twisting $\eta$ gives an identification $\Sh^\eta(X) \simeq \Sh(X)$. 
Let $X, Y$ be manifolds and $\eta_X \colon X \to B\Pic(\bfk)$ (resp.\ $\eta_Y \colon Y \to B\Pic(\bfk)$) be a twisting. 
For a morphism of manifolds $f \colon X \to Y$, if $f^*\eta_Y \coloneqq \eta_Y \circ f = \eta_X$, one can define functors\footnote{In this paper, we use the symbol $f^*$ instead of $f^{-1}$.}
\[
    f_*, f_! \colon \Sh^{\eta_X}(X) \to \Sh^{\eta_Y}(Y), \quad 
    f^*, f^! \colon \Sh^{\eta_Y}(Y) \to \Sh^{\eta_X}(X)
\]
satisfying the adjunction properties $f^*\dashv f_*$ and $f_!\dashv f^!$. 
Moreover, for two twisting $\eta,\eta' \colon X \to B\Pic(\bfk)$, we can define functors 
\begin{align*}
    \otimes \colon \Sh^{\eta}(X) \times \Sh^{\eta'}(X) & \to \Sh^{\eta \cdot \eta'}(X), \\ 
    \cHom \colon \Sh^{\eta}(X)^{\opp} \times \Sh^{\eta'}(X) & \to \Sh^{\eta^{-1} \cdot \eta'}(X).
\end{align*}
For an object $F \in \Sh^\eta (X)$, we can define its conic microsupport $\CMS(F)$ in a similar way to the untwisted case. 
We define $\Sh^\eta_A(X)$ in a similar way to the untwisted case.

We recall some facts about the microlocalization (see~\cite[Chap.~IV]{KS90}), in the twisted case. 
Let $\eta_1,\eta_2\colon X\to B\Pic(\bfk)$ be two twistings and let $F\in \Sh^{\eta_1}(X)$ and $G\in \Sh^{\eta_2}(X)$.
One can define a twisted sheaf $\muhom(F,G)\in \Sh^{\eta_1^{-1}\cdot \eta_2}(T^*X)$ on $T^*X$ in a similar way to \cite[Section~4.4]{KS90}, where $\eta_1^{-1}\cdot \eta_2\colon T^*X\to B\Pic (\bfk)$ is the composite of the projection $T^*X\to X$ and the twisting $\eta_1^{-1}\cdot \eta_2\colon X\to B\Pic(\bfk)$. 
Indeed, since the original $\muhom$ is defined via 6-functors, we can apply the same construction by tracing the twisting. 
The support of $\muhom(F,G)$ is contained in $\CMS(F) \cap \CMS(G)$.  
We have a natural isomorphism $\cHom(F,G) \simto \pi_* \muhom(F,G)$, and also $\cHom(F,G) \simeq i^*\muhom(F,G)$, where $i$ is the inclusion of the zero-section.

Now we assume that $\Lambda = \CMS(F) \setminus 0_X$ is a (conic) connected Lagrangian submanifold of $ T^*X\setminus 0_X$.  
For a function $f \colon X \to \bR$ of class $C^2$ such that $\Gamma_{df}$ intersect $\Lambda$ transversally at $(x_0;\xi_0)$, the space $m(F,f,x_0) = (\Gamma_{\{f\geq f(x_0)\}}(F))_{x_0}$ is called the \emph{microstalk} at $(x_0;\xi_0)$. 
It is proved that $m(F,f,x_0)$ is independent of $f$ and $(x_0;\xi_0)$ up to shift (see~\cite{KS90} Prop.~7.5.3 and Cor.~7.5.7). 
We say that $F$ is \emph{simple} or of microlocal rank $1$ along $\Lambda$ if $m(F,f,x_0) \simeq \bfk[d]$ for some $d \in \bZ$.

Let $F, G\in \Sh^\eta(X)$ be simple sheaves and assume that $\CMS(F)$ and $\CMS(G)$ intersect cleanly outside the zero-section. 
Then, for a connected component $\Lambda_0$ of $(\CMS(F)\cap \CMS(G))\setminus 0_X$, we have an isomorphism $\muhom (F,G)|_{\Lambda_0} \simeq \bfk_{\Lambda_0}[d]$ for some $d \in \bZ$. 

\subsection{Twisted Tamarkin category}

In this subsection, we introduce a twisted version of the Tamarkin category.
We follow \cite{KSZ23} for the $\infty$-categorical treatment of the Tamarkin category. 
We replace the Tamarkin direction $\bR_t$ with $\bR_t/\pi \bZ$, where $\pi$ is the area of the domain bounded by the standard unit circle $\cstd$ in $\bR^2$ with radius $1$.

Let $N$ be a manifold.
We fix a twisting $\eta \colon \bR_t/\pi\bZ \to B \Pic (\bfk)$. 
Since we work on $\bfk=\bF_2$, we may assume that $\eta$ is the delooping of $\bZ \to \Pic(\bfk); 1\mapsto \bfk[n]$ for some $n \in \bZ$.
By abuse of notation, we also write $\eta$ for the composite of $\eta$ and the projection $N \times \bR_t/\pi\bZ \to \bR_t/\pi\bZ$.

We consider the category $\Sh^\eta (N \times \bR_t/\pi\bZ)$ consisting of sheaves on $N \times \bR_t/\pi\bZ$ twisted by $\eta$. 
We define the twisted version of the Tamarkin category by 
\[
    \Tam(T^*N)\coloneqq \Sh^\eta (N \times \bR_t/\pi\bZ)/\{ F\mid \CMS(F)\subset \{\tau \le 0\}\}.
\]
The quotient functor $\Sh^\eta (N \times \bR_t/\pi\bZ)\to \Tam(T^*N)$ admits a left adjoint and a right adjoint. 
Both of these functors are fully faithful. 
We sometimes regard $\Tam (T^*N)$ as a full subcategory of $\Sh^\eta (N \times \bR_t/\pi\bZ)$ via either of these functors.
For an object $F \in \Tam(T^*M)$, we define 
\[
    \DMS(F) \coloneqq \CMS(F) \cap \{\tau=1\}.
\]
For a closed subset $A\subset T^*N\times \bR_t/\pi\bZ$, we set
\[
    \Tam_A(T^*N) \coloneqq \{ F \in \Tam(T^*N) \mid \DMS(F) \subset A \}.   
\]
We set $T^*_{\tau>0}(N \times \bR_t/\pi\bZ) \coloneqq \{ \tau>0\} \subset T^*(N \times \bR_t/\pi\bZ)$ and define a map $\rho \colon T^*_{\tau>0}(N \times \bR_t/\pi\bZ) \to T^*N$ by $(x,t;\xi,\tau) \mapsto (x;\xi/\tau)$.
For an object $F \in \Tam(T^*N)$, we set
\[
    \MS(F) \coloneqq \overline{\rho(\CMS(F) \cap \{\tau>0\})} \subset T^*N
\]
and call it the (non-conic or reduced) \emph{microsupport} of $F$.

Let $q_i \colon N \times \bR/\pi\bZ \times \bR/\pi\bZ \to N \times \bR_t/\pi\bZ; (x,t_1,t_2) \mapsto (x,t_i)$ denote the projection and $m \colon N \times \bR/\pi\bZ \times \bR/\pi\bZ \to N \times \bR_t/\pi\bZ; (x,t_1,t_2) \mapsto (x,t_1+t_2)$ denote the addition map.
For $F,G \in \Tam(T^*N)$, we define 
\begin{align*}
    F \star G & \coloneqq m_!(q_1^*F \otimes q_2^*G) \in \Tam(T^*N), \\
    \cHom^\star(F,G) & \coloneqq {q_1}_* \cHom(q_2^*F,m^!G) \in \Tam(T^*N).
\end{align*}
Then $\star$ induces the monoidal operation of $\Tam(T^*N)$, and $\cHom^\star$ defines the internal hom of $\Tam(T^*N)$. 

For $a \in \bR$, let $T_a$ be the map $N\times \bR_t/\pi\bZ \to N\times \bR_t/\pi\bZ \colon (x,t)\mapsto (x, t+[a])$, where $[a]$ is the image of the quotient map $\ell \colon \bR_t \to \bR_t/\pi\bZ$.
By definition, ${T_a}_*$ is a functor from $\Sh^{T_{a}^*\eta}(N\times \bR_t/\pi\bZ)$ to $\Sh^{\eta}(N\times \bR_t/\pi\bZ)$. 
We identify $\Sh^{T_{a}^*\eta}(N\times \bR_t/\pi\bZ)$ with $\Sh^{\eta}(N\times \bR_t/\pi\bZ)$ by the homotopy $(\eta\circ T_{sa})_{s \in [0,1]}$. 
We obtain an automorphism on $\Sh^{\eta}(N\times \bR_t/\pi\bZ)$ and it induces an automorphism on $\Tam(T^*N)$. 
We write the functor as $T_a$. 
Note that $T_0=\id$ and $T_{\pi}$ is the shift functor $[-n]$.

The functor $T_a$ is naturally isomorphic to the functor $\ell_!\bfk_{N\times [a,\infty)}\star (\mathchar`-)$. 
For $a\le a' \in \bR$, a natural transformation $\tau_{a,a'} \colon T_a \Rightarrow T_{a'}$ is induced by the natural morphism $\bfk_{N\times [a,\infty)}\to \bfk_{N\times [a',\infty)}$. 
This enable us to define a pseudo-distance $d$ on the set of the objects of $\Tam(T^*N)$ as in \cite{AI20,AISQ,AI24completeness}.
Namely, for $F, G \in \Tam(T^*N)$, define
\[
    d(F,G) \coloneqq 
    \inf \lc a+b \relmid 
    \begin{aligned}
        & \exists \alpha \colon F \to T_{a} G, \exists \beta \colon G \to T_{b} F \text{ such that } \\
        & T_{a} \beta \circ \alpha \simeq \tau_{0,a+b}(F), T_{b} \alpha \circ \beta \simeq \tau_{0,a+b}(G)
    \end{aligned} \rc.
\]
Such a pair of morphisms $(\alpha,\beta)$ is called \emph{$(a,b)$-interleaving} for $(F,G)$ and the pseudo-distance $d$ is called the \emph{interleaving distance}.
This pseudo-distance is in fact complete.

\begin{proposition}[{\cite[Cor.~4.5]{AI24completeness} and \cite[Prop.~6.22]{GV24coisotropic}}]\label{prop:completeness}
    The interleaving distance $d$ is a complete pseudo-distance, i.e., any Cauchy sequence with respect to $d$ has a limit object in $\Tam(T^*N)$.
\end{proposition}

The conic microsupport of a limit object can be estimated as follows.

\begin{proposition}[{\cite[Prop.~6.26]{GV24coisotropic}}]\label{prop:limit_ss}
    Let $(F_n)_n$ be a sequence in $\Tam(T^*N)$ and assume that it converges to $F_\infty$ with respect to the interleaving distance $d$. 
    Then 
    \[
        \DMS(F_\infty) \subset \bigcap_{k \in \bN} \overline{\bigcup_{n \ge k} \DMS(F_n)}.
    \]
\end{proposition}

The interleaving distance $d$ is degenerate in general, but it is proved that $d$ is non-degenerate on the category of metric-limit objects of constructible sheaves.
For a real analytic manifold $N$, an object $F \in \Tam(T^*N)$ is said to be \emph{limit constructible} if it is a metric limit of constructible sheaves with respect to the interleaving distance $d$.
A limit object of a sequence of limit constructible sheaves is unique up to isomorphism due to the following proposition. 

\begin{proposition}[{\cite[Prop.~B.8]{GV24coisotropic}}]\label{prop:limit_constructible}
    If $F, G \in \Tam(T^*N)$ are limit constructible and $d(F,G)=0$, then $F \simeq G$. 
\end{proposition}

We have the following isomorphism:
\begin{equation}\label{eq:isom_homstar}
    \Hom (F,T_a G)\simeq \Gamma_{[-a,\infty )}(\bR; \ell^!q_*\cHom^\star (F,G)),
\end{equation}
where $q \colon N \times \bR_t/\pi\bZ \to \bR_t/\pi\bZ$ is the projection.
We denote by $\cT(T^*N)$ the usual Tamarkin category of $T^*N$ defined as 
\[
    \cT(T^*N) \coloneqq \Sh(N \times \bR_t)/\{F \mid \CMS(F) \subset \{\tau \le 0\}\}.
\]
Then the functor $\ell^! \colon \Tam(T^*N) \to \cT(T^*N)$ is conservative and the functor $\ell_! \colon \cT(T^*N) \to \Tam(T^*N)$ is symmetric monoidal. 
One can equip a complete pseudo-distance $d$ with $\cT(T^*N)$ in a similar way to $\Tam(T^*N)$, and obtain a conic microsupport estimate similar to \cref{prop:limit_ss}. One can also define limit constructible objects in $\cT(T^*N)$ similarly.

A constructible object in $\cT(\pt)$ is isomorphic to $\bigoplus_{\alpha \in A} \bfk_{I_\alpha}[d_\alpha]$ for a locally finite family of intervals $(I_\alpha)_{\alpha \in A}$ and a family of integers $(d_\alpha)_{\alpha \in A}$ (\cite[Thm.~1.17]{KS18persistent} and \cite[Cor.~IV.4.3]{Gui23}).
For a limit constructible object in $\cT(\pt)$, we have the following decomposition by interval modules.

\begin{proposition}[{\cite[Cor.~B.12]{GV24coisotropic}}]\label{prop:barcode_decomposition}
    Let $F \in \cT(\pt)$ and assume that $F$ is limit constructible.
    Then there exists a countable family of intervals $(I_\alpha)_{\alpha \in A}$ and a family of integers $(d_\alpha)_{\alpha \in A}$ such that
    \[
        F \simeq \bigoplus_{\alpha \in A} \bfk_{I_\alpha}[d_\alpha].
    \]
    Moreover, for any $\varepsilon >0$, the family $(I_\alpha \mid \alpha \in A, |I_\alpha| \ge \varepsilon )$ is locally finite.
\end{proposition}

When $N=\pt$, we simply write $\Tam \coloneqq \Tam(\pt)$.
Similar to \cite[Prop.~5.5]{KSZ23} combined with \cite[Lem.~2.9]{CKNS24}, one can check 
\[
    \Tam(T^*N)\simeq \Sh(N)\otimes \Tam\simeq \Sh(N; \Tam),
\]
where the last category stands for the category of sheaves on $N$ with coefficient in $\Tam$.
Through this identification, the operations $\star$ and $\cHom^\star$ in the category $\Tam(T^*N)$ are usual $\otimes$ and $\cHom$ with coefficient in $\Tam$.
See \cite{Volpe} for 6-functor formalism for locally compact Hausdorff spaces and more general coefficients.

For $K_{12} \in \Tam(T^*(N_1 \times N_2)), K_{23} \in \Tam(T^*(N_2 \times N_3))$, we can also define the operation $\ostar$ by 
\[
    K_{12} \ostar K_{23} \coloneqq m_{13!} (q_{12}^*K_{12} \otimes q_{23}^*K_{23}),
\]
where $q_{ij} \colon N_1 \times N_2 \times N_3 \times \bR/\pi\bZ \times \bR/\pi\bZ \to N_i \times N_j \times \bR_t/\pi\bZ$ is the projection, and 
\[
\begin{aligned}
    m_{13} \colon N_1 \times N_2 \times N_3 \times \bR/\pi\bZ \times \bR/\pi\bZ & \to N_1 \times N_3 \times \bR/\pi\bZ; \\
    (x_1,x_2,x_3,t_1,t_2) & \mapsto (x_1,x_3,t_1+t_2).
\end{aligned}
\]
Through the identification with sheaf category with coefficient in $\Tam$, the operation $\ostar$ corresponds to the usual convolution. 

For $K_{12} \in \cT(T^*(N_1 \times N_2)), K_{23} \in \Tam(T^*(N_2 \times N_3))$, we can also define the operation $\ostar$ by a similar method. 
This $K_{12}\ostar K_{23}$ is isomorphic to $\ell_! K_{12} \ostar K_{23}$ defined above. 

\begin{remark}\label{rem:twisted_sheaves}
    The category $\Tam(T^*N)$ can be identified with a full subcategory of $\Sh^\eta(N\times \bR_t/\pi\bZ)$. 
    We can describe objects of $\Sh^\eta(N\times \bR_t/\pi\bZ)$ via untwisted sheaves. 
    Take real numbers $t_0<t_1<t_2<t_3$ satisfying $t_2-t_0<\pi$, $t_3-t_1<\pi$, and $t_3-t_0>\pi$. 
    Set $U_0=\ell((t_0, t_2))$, $U_0=\ell((t_1, t_3))$, $V_0=\ell((t_0, t_1))$, and $V_1=\ell((t_2, t_3))$. 
    By the sheaf property of $\Sh^\eta(\mathchar`-)$ on $M\times \bR_t/\pi \bZ$, 
    an object $\Sh^\eta(N\times \bR_t/\pi\bZ)$ is equivalent to the datum $(F_0, F_1, \alpha_1, \alpha_0)$ where $F_i$ is an object of $\Sh(N\times U_i)$ for each $i=0,1$, and $\alpha_1 \colon F_0|_{N\times V_1}\simeq F_1|_{N\times V_1}$, $\alpha_0 \colon F_0[n]|_{N\times V_0}\simeq F_1|_{N\times V_0}$ are isomorphisms. 

    Gluing $F_0$ and $F_1$ by $\alpha_1$ firstly, we can see that above datum is also equivalent to $(F, \alpha_0)$ where $F$ is an object of $\Sh(N\times (t_0,t_3))$
    and $\alpha_0 \colon F[n]|_{N\times (t_0,t_3-\pi)}\simeq F|_{N\times (t_0+\pi, t_3)}$ is an isomorphism via the identification $N\times (t_0,t_3-\pi)\simeq N\times (t_0+\pi, t_3)\colon (x,t)\mapsto (x,t+\pi)$. 
\end{remark}

For $F,G\in \Tam (T^*N)$, the object $\muhom (F,G)|_{\{\tau>0\}}\in \Sh(\{\tau>0\})$ is invariant under isomorphisms in $\Tam(T^*M)$. 
Not only $\muhom|_{\{\tau>0\}}\colon\Tam (T^*N)^{\opp}\times \Tam (T^*N) \to \Sh(\{\tau>0\})$ is a functor, but also $\muhom$ makes $\Tam (T^*N)$ into a $\Sh(\{\tau>0\})$-enriched category. 
This follows from the fact that $\muhom$ is the Hom sheaf of a stack called Kashiwara--Schapira stack~\cite{KS90,Gui23}. 
See also \cite[Remark~2.13]{KuoLiSpherical} for an $\infty$-categorical treatment. 
In what follows, we denote $\muhom(F,G)|_{\{\tau>0\}}$ simply by $\muhom(F,G)$ for $F,G \in \Tam(T^*N)$.

We have the following (co)fiber sequence associated with the Hom spaces and $\muhom$.

\begin{lemma}\label{lem:muhom_fiber_seq}
    For $F, G \in \Tam(T^*N)$ such that $\MS(F)$ and $\MS(G)$ are compact, we have a fiber sequence 
    \[
        \colim_{\varepsilon \to 0} \Hom(F, T_{-\varepsilon}G) \to \Hom(F,G) \to \Gamma(\{\tau>0\};\muhom(F,G)). 
    \]
\end{lemma}
\begin{proof}
    Let $\cH \coloneqq \ell^! q_*\cHom^\star(F,G)$. 
    By a similar argument to \cite{Ike19}, we have an isomorphism 
    \[
        \Gamma_{[0,\infty)}(\cH)_0 \simeq \Gamma(\{\tau>0\};\muhom(F,G)),
    \]
    where we use the compactness assumption.
    For $\varepsilon>0$, we have a fiber sequence
    \[
        \Gamma_{[\varepsilon,\infty)}(\bR;\cH) \to \Gamma_{[0,\infty)}(\bR;\cH) \to \Gamma_{[0,\varepsilon)}((-\infty,\varepsilon);\cH).
    \]
    By \eqref{eq:isom_homstar}, the first term is isomorphic to $\Hom(F,T_{-\varepsilon}G)$ and the second term is isomorphic to $\Hom(F,G)$. 
    Thus, by taking colimit as $\varepsilon \to 0$, we obtain the result.
\end{proof}

\subsection{Hamiltonian action}

Let $H \colon T^*N \times I \to \bR $ be a $C^\infty$-function with compact support. 
Denote by $\phi^H=(\phi^H_s)_{s \in I} \colon T^*N \times I \to T^*N$ be the associated Hamiltonian isotopy. 
It is proved in \cite{GKS} that there exists an object $K(\phi^H) \in \Sh((N \times \bR)^2 \times I)$ whose conic microsupport outside the zero-section is equal to the conic Lagrangian movie associated with the graph of $\phi^H$.
The push forward by the map $(N \times \bR)^2 \times I \to N^2 \times \bR \times I; (x_1,t_1,x_2,t_2,s) \mapsto (x_1,x_2,t_1-t_2,s)$ and the quotient morphism $\Sh(N^2 \times \bR \times I) \to \cT(T^*(N^2 \times I))$, the object $K(\phi^H)$ defines $\cK(\phi^H) \in \cT(T^*(N^2 \times I))$, which is called the sheaf quantization or the \emph{GKS kernel} of $\phi^H$.

With a time-independent non-negative $C^\infty$-function $H \colon T^*N \to \bR$ with non-compact support, we can associate an object $\cK(\phi^H)_1 \in \cT(T^*N^2)$ as follows. 
We take a sequence of compact subset $(K_n)_n$ such that $\bigcup_n \Int(K_n)=T^*N$ and a sequence of cutoff functions $(\chi_n \colon T^*N \to [0,1])_n$ of class $C^\infty$ such that $H_n|_{K_n} \equiv 1$, $\supp(H_n) \subset \Int(K_{n+1})$, and  $H_n \le H_{n+1}$. 
Then $H_n \coloneqq \chi_n \cdot H$ has a compact support, and thus defines $\cK(\phi^{H_n}) \in \cT(T^*(N^2 \times I))$. 
By \cite{GKS}, we have a canonical continuation morphism $\cK(\phi^{H_n})_1 \to \cK(\phi^{H_{n+1}})_1$ in $\cT(T^*N^2)$ and define 
\[
    \cK(\phi^H)_1 \coloneqq \colim_{n} \cK(\phi^{H_n})_1 \in \cT(T^*N^2).
\]

Let $\varphi \in \Ham_c(T^*N)$ be a compactly supported Hamiltonian diffeomorphism on $T^*N$. 
For a compactly supported $C^\infty$-function $H \colon T^*N \times I \to \bR$ such that $\phi^H_1=\varphi$, the object $\cK(\phi^H)|_{s=1}$ does not depend on the choice of~$H$ (see \cite{AI24completeness}), which we will denote by $\cK(\varphi) \in \cT(T^*N^2)$.
We call $\cK(\varphi)$ the sheaf quantization or the \emph{GKS kernel} of $\varphi$.

Recall that we set $\bfk=\bF_2$.
In this case, it is proved in \cite{GV24coisotropic} that the distance $d(\cK(\varphi_0), \cK(\varphi_1))$ is equal to the spectral metric between $\varphi_0$ and $\varphi_1$: 
\[
    d(\cK(\varphi_0), \cK(\varphi_1)) = \gamma(\varphi_0,\varphi_1). 
\]
By \cite{seyfaddini2012descent}, for a fixed compact subset $K$ of $T^*N$, there exists a constant $C'>0$ such that for any $\varphi_0,\varphi_1$ whose supports are contained in $K$, 
\[
    \gamma(\varphi_0,\varphi_1) \le C' d_{C^0}(\varphi_0,\varphi_1).
\]
By combining these results, we obtain 
\[
    d(\cK(\varphi_0), \cK(\varphi_1)) \le C' d_{C^0}(\varphi_0,\varphi_1)
\]
for any $\varphi_0,\varphi_1$ whose supports are contained in $K$.
Since $\cT(T^*N^2)$ is complete with respect to the pseudo-distance $d$ (\cref{prop:completeness}), for any compact supported Hamiltonian homeomorphism $\varphi$ on $T^*N$, we obtain an object $\cK(\varphi) \in \cT(T^*N^2)$ whose microsupport is the graph of $\varphi$ by \cref{prop:limit_ss}.
If there is no confusion, we simply write $\cK(\varphi) F$ for $\cK(\varphi) \ostar F$.

\begin{lemma}\label{lem:T-linear}
    Let $F,G \in \Tam(T^*N)$ and $\varphi$ be a Hamiltonian homeomorphism with compact support on $T^*N$.
    Then, one has 
    \[
    q_*\cHom^\star(\cK(\varphi) F, \cK(\varphi) G)\simeq q_*\cHom^\star(F, G)
    \]
\end{lemma}
\begin{proof}
    Under the identification $\Tam (T^*N)\simeq \Sh (N;\Tam)$, 
    $q_*\cHom^\star$ is the $\Tam$-enriched hom space. 
    Then the result follows since $\cK(\varphi)\ostar (\mathchar`- )$ is a $\Tam$-linear equivalence. 
\end{proof}

\section{Sheaf quantization associated with Jordan curves}\label{section:SQ}

In what follows, until the end of this paper, we set $M=\bR_x$. 

\subsection{Sheaves associated with the torus}

In \cite{AISQ}, the authors constructed small sheaf quantizations for a class of rational Lagrangian immersions following the idea of Guillermou~\cite{Gu12,Gui23}. 
Here, we apply the sheaf quantization method to the standard unit circle $\cstd$ in $T^*\bR_x \simeq \bR^2$ in a more sophisticated way. 
The outcome can be seen as a sheaf quantization of $\cstd \times \cstd$ in $T^*\bR^2=T^*(\bR_{x_1}\times \bR_{x_2})$. 
In particular, instead of the orbit category, we use the category of twisted sheaves, which was introduced in the previous section. 
This can be done because of the monotonicity of Lagrangian submanifolds that we will handle. 

The idea to construct a sheaf quantization of $\cstd \times \cstd$ as a sheaf on $M \times M \times \bR_t/\pi \bZ$ without another extra $\bR$-factor is due to St\'{e}phane Guillermou.
This makes all the computation much easier.

Set $L=\cstd$ to be the standard circle with center $(0,0)$ and radius $1$. 
Since the space of 1-jet $J^1(M)=T^*M\times \bR_t$ has a natural contact structure that is invariant with respect to the translation in the $\bR_t$-direction, the quotient $T^*M\times \bR_t/\pi \bZ$ inherits a natural contact structure. 
We define a primitive of $\cstd$ valued in $\bR/\pi\bZ$ by $\fstd(s) \coloneqq \frac{1}{2}s-\frac{1}{4}\sin 2s$.
We take a Legendrian lift $\tl{L}$ in $T^*M\times \bR_t/\pi \bZ$ of $L=\cstd$ as follows:
\[
    \tl{L}=\left\{ \lb (\cos s ; \sin s) , -\fstd(s) \rb \in T^*M\times \bR_t/ \pi \bZ  \relmid s\in \bR/2\pi \bZ \right\}.
\]
We also define a Legendrian lift $\Lambda \subset T^*M \times T^*M \times \bR_t/ \pi \bZ $ of $L \times L\subset T^*M^2$ by 
\begin{equation}\label{eq:Lambda}
    \Lambda=\{ ( (\cos s_1 ; \sin s_1) ,  (\cos s_2 ; \sin s_2) , -\fstd(s_1) -\fstd(s_2))  \mid s_1, s_2\in \bR_t/2\pi \bZ \}. 
\end{equation}
We identify $T^*M \times \bR_t/\pi \bZ$ with the subset $\{ \tau=1 \}$ in $T^*(M \times \bR/\pi \bZ)$ as contact manifolds.

Below we will prove the following.

\begin{proposition}\label{prop:existence_SQ}
    There exists a simple object $\Fstd \in \Tam(T^*M^2)$ such that $\DMS(\Fstd)=\Lambda$ and such an object is unique up to degree shift. 
    Moreover, $\Hom(\Fstd,\Fstd) \simeq H^*(S^1)$. 
\end{proposition}

We can define the Kashiwara--Schapira stack $\mush_{\tl{L}}$ on $\tl{L}$, which is regarded as a subset of $\{\tau=1\} \subset T^*(M \times \bR_t/\pi\bZ)$. 
This stack is locally isomorphic to the stack of local systems, but globally it is twisted. 
In our setting, this twisting is the delooping of $\bZ \to \Pic(\bfk) \colon 1\mapsto \bfk[2]$, which corresponds to the first Maslov class of $L$.
We write $\eta^{-1}$ for the twisting $L \to B\Pic (\bfk)$ and write $\eta$ for its inverse.
Then we have an isomorphism of stacks $\mush_{\tl{L}} \simeq \Loc_{\tl{L}}^{\eta^{-1}}$, where the right-hand side denotes the stack of local systems with twisting $\eta^{-1}$.
By twisting, we have an isomorphism $\mush_{\tl{L}}^{\eta} \simeq \Loc_{\tl{L}}$, which has a global object.

\begin{remark}
    As explained in \cite{JT17}, the twisting $\eta^{-1} \colon L \to B\Pic(\bfk)$ is described as the composite of the Gauss map, (the delooping of) the $J$-homomorphism, and the morphism induced by the unit morphism $\bS \to H\bfk$, where $\bS$ denotes the sphere spectrum and $H\bfk$ denotes the Eilenberg--MacLane spectrum. 
    In order to get the above isomorphism, we need to choose a homotopy between $L\to U/O \to B\Pic(\bS) \to B\Pic (\bfk)$ and $\eta^{-1}$. 
    The connected components of the space of such homotopies forms a $\bZ$-torsor, and each component is contractible. 
    We can freely choose a connected component for the following argument. 
    The differences of the choices affect as overall degree shifts. 
\end{remark}

The twisting $\eta \colon \tl{L} \simeq L \to B\Pic (\bfk)$ factors through the base space $M \times \bR_t/\pi \bZ$. 
Since the projection $\pi \colon \tl{L}\to M\times \bR_t/\pi \bZ$ is of finite position, we can apply the doubling method (with cusp doubling, which is used in \cite{NadlerShende20,GPS24,IkeKuwagaki2023}) by Guillermou to obtain a morphism of stacks on $M \times \bR_t/\pi\bZ$:
\[
    \pi_*\mush^\eta_{\tl{L}}\to \Sh^{\eta}_\Lambda ((\mathchar`-)\times (-1,-1+\varepsilon) )
\]
for sufficiently small $\varepsilon>0$.
Here, the right-hand side denotes the stack defined as $U \mapsto \Sh^{\eta}_{\Lambda \cap T^*(U\times  (-1,-1+\varepsilon))} (U \times (-1,-1+\varepsilon))$ for an open subset $U$ of $M \times \bR_t/\pi\bZ$.

For $x\in (-1,1)$, the set $\Lambda_x \coloneqq \pi_{\xi_2}(\Lambda \cap \{ x_2= x\})\subset T^*M\times \bR_t/\pi\bZ$ consists of the two copies of $\tl{L}$ shifted to the $\bR_t$-direction. 
There exists a contact isotopy $(\psi_x)_{x\in (-1,1)}$ on $T^*M\times \bR_t/\pi\bZ$ such that $\psi_0=\id$ and $\psi_x(\Lambda_0)=\Lambda_x$. 
By applying the GKS kernels associated with the contact isotopy, we obtain an isomorphism 
\[
    \Sh^{\eta}_\Lambda (M\times \bR_t/\pi\bZ \times (-1,-1+\varepsilon) )\simeq \Sh^{\eta}_\Lambda (M\times \bR_t/\pi\bZ \times (-1,1)).
\]

Let $\cL$ be a global object of $\Loc_L$. 
By sending $\cL$ through the identification $\mush_{\tl{L}}^{\eta} \simeq \Loc_{\tl{L}}$ and the morphism above, we obtain a sheaf quantization $G_{\cL, \cstd} \in \Sh^{\eta} (M\times \bR_t/\pi \bZ \times (-1,1) )$.
Denote by $j \colon (-1,1) \hookrightarrow M=\bR_{x_2}$ the inclusion and also write $j$ for the base change $M\times \bR_t/\pi \bZ \times (-1,1)\to M\times M\times \bR_t/\pi \bZ$. 
By pushing forward under $j$, we obtain an object $F_{\cL, \cstd} \coloneqq j_!G_{\cL, \cstd}$. 

\begin{lemma}
    One has $j_!G_{\cL, \cstd} \simeq j_*G_{\cL, \cstd}$, and they are objects of $\Tam_{\Lambda}(T^*M^2)$. 
\end{lemma}
\begin{proof}
    By the construction $j_*G_{\cL, \cstd}|_{\{x_2=-1\}}$ is $0$. Let $i$ be the inclusion $M\times \{1\}\times \bR_t/\pi\bZ\to M^2 \times \bR_t/\pi\bZ$. 
    There is a cofiber sequence $j_!G_{\cL, \cstd} \to j_*G_{\cL, \cstd} \to i_*i^*j_*G_{\cL, \cstd}$. 
    The (conic) microsupport estimates shows \linebreak $\CMS(i^*j_*G_{\cL, \cstd})\subset T_{\frac{\pi}{2}}\tl{L}$, and hence a similar estimate for $\CMS(\ell^!i^*j_*G_{\cL, \cstd})$ holds since $\ell$ is a submersion.
    By \cite[Proposition~5.8]{STZ17}, $\ell^!i^*j_*G_{\cL, \cstd}$ must be a local system and becomes $0$ in $\cT (T^*M)$. 
    Since $\ell^!$ is conservative, $i^*j_*G_{\cL, \cstd}$ is also $0$.
    By estimating the both sides of $\CMS(j_!G_{\cL, \cstd})=\CMS(j_*G_{\cL, \cstd})$, we find that $F_{\cL, \cstd} \in \Tam_{\Lambda}(T^*M^2)$.
\end{proof}

We set $\Fstd \coloneqq F_{\underline{\bfk}, \cstd}$, where $\underline{\bfk}$ denotes the trivial local system of rank $1$ on $L$.

\begin{lemma}\label{lem:SQ_ff}
    The functor $\Loc_{\tl{L}}(\tl{L})\to \Tam_{\Lambda}(T^*M^2)$ is fully faithful. 
\end{lemma}
\begin{proof}
    By \cite{NadlerShende20}, the functor $\Loc_{\tl{L}}(\tl{L})\to \Sh^{\eta}_\Lambda (M\times \bR_t/\pi\bZ \times (-1,-1+\varepsilon) )$ is fully faithful. 
    As the composite, the functor $\Loc_{\tl{L}}(\tl{L})\to \Sh^{\eta}_\Lambda (M\times \bR_t/\pi\bZ \times (-1,1) )$ is also fully faithful. 
    One can check that the image of the functor is in $\Tam(T^*N)$, which is regarded as a subcategory of $\Sh^\eta(N \times \bR_t/\pi\bZ)$.
    Since $\End(j_!G) \simeq \Hom(G, j^!j_!G) \simeq \End(G)$, the functor $j_!$ is also fully faithful.
    By combining these, we obtain the result.
\end{proof}

By \cref{lem:SQ_ff}, we have
\[
    \Hom_{\Tam(T^*M^2)}(\Fstd,\Fstd) \simeq \Hom_{\Loc_L}(\underline{\bfk}, \underline{\bfk}) \simeq H^*(S^1), 
\]
where $\underline{\bfk}$ denotes the trivial local system of rank $1$ on $L$. 

We shall prove the uniqueness by decomposing the sheaf into easier pieces. 
A similar argument can be found in \cite[Part~VI]{Gui23}.

\begin{lemma}\label{lem:Fcstdunique}
    Simple objects in $\Tam_{\Lambda}(T^*M^2)$ are unique up to shift. 
\end{lemma}
\begin{proof}
    Let us first observe the image of the projection $\Lambda\to M^2\times \bR_t/\pi\bZ$. 
    The immersed locus is given by 
    \[
        ((\cos \pm s;\sin \pm s),(\cos \mp s, \sin \mp s),-f_0(\pm s)-f_0(\mp s)) \mapsto (\cos s, \cos s, 0)
    \]
    and
    \[
        ((\cos \pm s;\sin \pm s),(\cos \pi \mp s, \sin \pi \mp s),-f_0(\pm s)-f_0(\pi \mp s)) \mapsto (\cos s, -\cos s, \frac{\pi}{2}).
    \]
    We take $t_0, t_3$ in \cref{rem:twisted_sheaves} so that $-\pi/2< t_0<t_3-\pi <0$.
    We note that $\ell ((t_0,t_3-\pi))$ does not contain $0$ nor $\frac{\pi}{2}$. 
    
    We will see that simple sheaves on $M^2\times (t_0,t_3)$ with $\DMS \subset \ell^{-1}(\Lambda)\cap T^*M^2\times (t_0,t_3)$ that corresponds to an object of $\Tam(T^*M^2)$ are unique up to shift. 
    Let $F\in \Sh(M^2 \times (t_0,t_3))$ be such a sheaf.
    The support of $F$ on $M^2\times (t_0,t_3)$ is bounded since $F$ corresponds to an object in $\Tam_{\Lambda}(T^*M^2)$, which implies that the support is the union of the closures of three bounded regions. 
    We write $F$ as an extension of $F_{(t_0, 0)}, F_{[0, \pi/2)}$ and $F_{[\pi/2, t_3)}$. 
    Each of $F_{(t_0, 0)}, F_{[0, \pi/2)}, F_{[\pi/2, t_3)}$ is unique up to shift by the microsupport condition. 
    The non-trivial extension class is also unique. 
    
    The choice of an isomorphism $\alpha\colon F_{(t_0,t_3-\pi)}[-2] \simto T_{-\pi} F_{(t_0+\pi, t_3)}$ is also unique. 
    This proves the lemma.
\end{proof}
This completes the proof of \cref{prop:existence_SQ}. 

\begin{remark}\label{rem:unobstructed}
    The existence of sheaf quantization $\Fstd$ of $\cstd\times \cstd\subset T^*M^2$ in the category $\Tam(T^*M^2)$ would be related to the fact that $\cstd \times \cstd$ admits a bounding cochain \cite{FOOO09}. 
    In contrast, $\cstd\subset T^*\bR$ does not admit a bounding cochain and is unobstructed only modulo $T^\pi$ in the sense of \cite{FOOO09}, which would be why one can only construct a sheaf quantization of $\cstd$ in $\Sh(M \times (0,\pi) \times \bR_t/\pi \bZ)$ with the doubling parameter (cf.\ \cite{AISQ}).  
\end{remark}

\begin{corollary}\label{cor:pi_torsion}
    The natural morphism $\Fstd \to T_{\pi}\Fstd$ induced by the natural transformation $\id=T_0 \Rightarrow T_\pi$ is zero.
\end{corollary}
\begin{proof}
    Since 
    \[
        \Hom (\Fstd, T_\pi \Fstd)\simeq \End(\Fstd)[2]\simeq H^*(S^1)[2],
    \]
    we have $H^0(\Hom(\Fstd, T_\pi \Fstd))=0$.
\end{proof}

We shall describe the morphism 
\begin{equation}\label{eq:microlocalization_morphism}
    H^*(\cstd) \simeq \Hom(\Fstd,\Fstd) \\ 
    \to \Gamma(\{\tau>0\};\muhom(\Fstd,\Fstd))
    \simeq H^*(\cstd \times \cstd).
\end{equation}
The generator $v \in H^1(\cstd)$ is sent to $v\otimes 1+1\otimes v \in H^1(\cstd\times \cstd)$. 
Indeed, we find that the coefficient of $v\otimes 1$ is non-trivial by construction, and that of $1\otimes v $ by symmetry with respect to $(z_1,z_2) \mapsto (z_2,z_1)$.
\medskip

Let $\phi$ be a Hamiltonian diffeomorphism with compact support on $T^*M$, and denote by $\cK(\phi \times \phi) \in \cT(T^*M^4)$ the sheaf quantization of $\phi \times \phi$. 
Then the composition with $\cK(\phi \times \phi)$ induces a $\Tam$-linear autoequivalence of the category $\Tam(T^*M^2)$. 
Moreover, we have $\MS(\cK(\phi \times \phi) F) = (\phi \times \phi)(\MS(F))$ for any $F \in \Tam(T^*M^2)$.
Thus, $F_{\cur} = \cK(\phi \times \phi)\Fstd$ is a sheaf quantization for $\cur \times \cur=(\phi\times\phi)(\cstd\times \cstd)$.

\subsection{Action of \texorpdfstring{$R_\theta$}{Rtheta}}

Now we consider the action of $R_\theta$ defined in \eqref{eq:Rtheta} on $\Tam(T^*M^2)$.
The Hamiltonian function of $R_\theta$ is the non-negative function $H \colon T^*M^2 \simeq \bC^2$ defined as $H(z_1,z_2)=|z_1-z_2|^2/4$. 
Hence, we can construct an object $\cK(\phi^H)_\theta \in \cT(T^*M^2)$ for any $\theta$.
By \cite{GKS}, we have continuation morphisms $\cK(\phi^H)_\theta \to \cK(\phi^H)_{\theta'} \ (\theta \le \theta')$.
By abuse of notation, we also write $R_\theta$ for $\cK(\phi^H)_\theta \ostar (\mathchar`- )$, the automorphism on $\Tam(T^*M^2)$.

The Hamiltonian function $H$ which generates the Hamiltonian isotopy $(R_\theta)_\theta$ also defines a contact isotopy $\tl{R}=(\tl{R}_\theta)_\theta$ on $\{\tau=1\}$. 
This isotopy $\tl{R}=(\tl{R}_\theta)_\theta$ is the product of $(R_\theta)_\theta$ and the identity morphism of $\bR_t/\pi\bZ$. 

\begin{lemma}\label{lem:R2pi}
    There is an isomorphism $\cK(\phi^H)_{2\pi}\simeq \bfk_{\Delta\times [0,\infty)}[2]$. 
    Hence, the functor $R_{2\pi}$ on $\Tam(T^*M^2)$ coincides with the degree shift $[2]$. 
\end{lemma}
\begin{proof}
    First we have the (conic) microsupport estimate $\CMS(\cK(\phi^H)_{2\pi})=\CMS(\bfk_{\Delta\times [0,\infty)})$.
    Moreover, $\cK(\phi^H)_{2\pi}$ is simple along its conic microsupport. 
    Since $\bfk=\bF_2$, we obtain $\cK(\phi^H)_{2\pi} \simeq \bfk_{\Delta\times [0,\infty)}[d]$ for some $d\in \bZ$. 
    We can observe the grading by tracing the action on the fiberwise universal covering space of the Lagrangian Grassmannian bundle of $T^*M^2$ as in \cite{Seidel00}. 
\end{proof}

\begin{remark}
    For our purpose, it is enough to cut off the support of $H$ outside a sufficiently large compact subset. 
    Then we only need sheaf quantization of Hamiltonian isotopies with compact support.
    From this position, the statement of \cref{lem:R2pi} should be understood as that the action of $R_{2\pi}$ on the objects whose microsupports are contained in the compact subset coincides with the degree shift $[2]$. 
\end{remark}

We can also determine $R_\pi \Fstd$ as follows.

\begin{lemma}
    One has an isomorphism
    \[
        R_\pi \Fstd \simeq \Fstd[1].
    \]
\end{lemma}
\begin{proof}
    Since $\CMS(R_\pi \Fstd) = \Lambda$, by the uniqueness in \cref{prop:existence_SQ}, we have $R_\pi \Fstd \simeq \Fstd[d]$ for some $d \in \bZ$.
    Then, by \cref{lem:R2pi},
    \[
        \Fstd[2] \simeq R_{2\pi} \Fstd \simeq R_\pi \Fstd[d] \simeq \Fstd[2d], 
    \]
    which concludes $d=1$.
\end{proof}

\subsection{Computation for the standard circle}\label{subsection:standard_circle}

Let $\Fstd \in \Tam(T^*M^2)$ be the sheaf quantization of the standard torus $\cstd \times \cstd$ constructed in \cref{prop:existence_SQ}.
We define 
\[
    \cV_{\cstd,\theta} \coloneqq \ell^! q_* \cHom^\star(\Fstd,R_\theta \Fstd) \in \cT(\pt ),
\]
where $q \colon M^2 \times \bR_t/\pi\bZ \to \bR_t/\pi\bZ$ and $\ell \colon \bR_t \to \bR_t/\pi\bZ$ are the projection and the quotient map.
It is also convenient to consider the family version 
\[
    \cV_{\cstd} \coloneqq (\ell\times \id_{[0,\pi]})^! (q\times \id_{[0,\pi]})_* \cHom^\star(q'^* \Fstd,R \Fstd)\in \Sh ([0,\pi];\cT(\pt )),
\]
where $R$ is the GKS kernel for the (full) Hamiltonian isotopy $(R_\theta)_{\theta\in [0,\pi]}$ and $q'\colon  M^2 \times \bR_t/\pi\bZ \times [0,\pi ]\to M^2 \times \bR_t/\pi\bZ $ is the projection. 
For $\theta_0\in [0,\pi]$, we have an isomorphism $\cV_{\cstd}|_{\{\theta=\theta_0\}}\simeq \cV_{\cstd,\theta_0}$. 

For each $\theta\in (0,\pi)$ and $a \in [-\pi,0]$, we can directly check that
\[
  T_{-a}\tl{R}_\theta(\Lambda)\cap \Lambda =
  \begin{cases}
    \lc ((\cos s;\sin s), (\cos s;\sin s), -2f_0(s)) \relmid s\in \bR/2\pi\bZ \rc  & (a=0)\\
    \lc ((\cos s;\sin s), (-\cos s ;-\sin s ), -2f_0(s)-\frac{\pi}{2}) \relmid s\in \bR/2\pi\bZ \rc  & (a=-\theta)\\
    \varnothing  & (\text{otherwise}).
  \end{cases}
\]

Decompose the strip $\bR \times [0,\pi]$ into locally closed isosceles right triangles as follows:
\[
\begin{aligned}
    \vartriangle_n&\coloneqq \{ (t,\theta)\mid n\pi \le t< n\pi +\theta \}, \\
    \triangledown_n&\coloneqq \{ (t,\theta)\mid (n-1)\pi +\theta \le t< n\pi \}.
\end{aligned}
\]
We also set
\[
\begin{aligned}
    \vartriangle'_n&\coloneqq \{ (t,\theta)\mid n\pi <  t\le n\pi +\theta \}, \\
    \triangledown'_n&\coloneqq \{ (t,\theta)\mid (n-1)\pi +\theta < t \le n\pi \}.
\end{aligned}
\]
By the microlocal Morse lemma and the intersection estimate above, we obtain the following:

\begin{proposition}\label{prop:standard_structure}
    If $(a_0, \theta_0)$ and $(a_1,\theta_1)$ belong to the same component of the decomposition by $\vartriangle'_n$'s and $\triangledown'_n$'s, then 
    \[
    \Hom (\Fstd, T_{-a_0} R_{\theta_0} \Fstd) \simeq \Hom (\Fstd, T_{-a_1} R_{\theta_1} \Fstd)
    \]
    as $\End (\Fstd)$-modules.
    If $(a, \theta)\in \vartriangle'_n$, we have 
    \[
        \Hom (\Fstd, T_{-a} R_{\theta} \Fstd)\simeq \Hom (\Fstd, T_{-(n+1)\pi} R_{\pi} \Fstd)\simeq \End (\Fstd)[-2n-1].
    \]
    If $(a, \theta)\in \triangledown'_n$, we have
    \[
        \Hom (\Fstd, T_{-a} R_{\theta} \Fstd)\simeq \Hom (\Fstd, T_{-n\pi}  \Fstd)\simeq \End (\Fstd)[-2n].
    \] 
\end{proposition}

It is not difficult to determine the whole structure of $\cV_{\cstd}$ and $\cV_{\cstd, \theta}$ as follows. 
We will not use the following proposition and omit the proof. 

\begin{proposition}
    One has an isomorphism
    \[
        \cV_{\cstd} \simeq \bigoplus_{n \in \bZ} \bfk_{\vartriangle_n\cup \triangledown_{n+1}}[-2n] \oplus \bigoplus_{n \in \bZ} \bfk_{\vartriangle_n\cup \triangledown_{n}}[-2n+1]. 
    \]
    For $\theta \in [0,\pi]$, one has an isomorphism
    \[
        \cV_{\cstd,\theta} \simeq \bigoplus_{n \in \bZ} \bfk_{[n\pi,(n+1)\pi)}[-2n] \oplus \bigoplus_{n \in \bZ} \bfk_{[\theta+(n-1)\pi,\theta+n\pi)}[-2n+1]. 
    \]
    Moreover, for any $a \in \bR$, the right action of $v \in H^1(S^1)$ on the stalk $(\cV_{\theta})_a$ is non-zero.
\end{proposition}

\section{Sheaf-theoretic condition for rectangular peg}\label{section:sheaf_condition}

In this section, we prove the following theorem. 
\begin{theorem}\label{thm:main_sheaves}
    Let $\phi$ be a Hamiltonian homeomorphism with compact support. 
    Let us consider the Jordan curve $\cur=\phi(\cstd)$. 
    Define $F_{\cur} \coloneqq \cK(\phi \times \phi) \Fstd$.
    If $T_a\DMS(F_{\cur})\cap \DMS(F_{\cur})=\varnothing$ for any $a \in \bR \setminus \pi \bZ$, then $\cur$ inscribes a $\theta$-rectangle for any $\theta \in (0,\pi)$.
\end{theorem}

Before starting the proof, we give its rough outline. 
We consider the persistence module $(\Hom(F_{\cur}, T_a R_\theta F_{\cur}))_{a \in \bR}$ (in the derived sense) with structure morphisms $(\tau_{a,a'})_{a \le a'}$. 
We focus on a ``critical value'' $a_0 \in \bR$ such that $\tau_{a,a'}$ is not an isomorphism if $a<a_0<a'$. 
We will prove: 
\begin{enumerate}
    \item[(A)] A critical value $a_0$ is produced by some subset in the intersection $(\cur \times \cur) \cap R_\theta(\cur \times \cur)$.

    \item[(B)] Under the assumption $T_a\DMS(F_{\cur})\cap \DMS(F_{\cur})=\varnothing$ for any $a \in \bR \setminus \pi \bZ$, a critical value produced by the diagonal $\Delta_{\cur}$ in the sense of (A) is in $\pi \bZ$.

    \item[(C)] There is a critical value $a_0$ in $\bR \setminus \pi \bZ$. 
\end{enumerate}
These three assertions prove the existence of a point in $(\cur \times \cur) \cap R_\theta(\cur \times \cur) \setminus \Delta_{\cur}$, which implies the existence of a $\theta$-rectangle on $\cur$.

By \cref{lem:muhom_fiber_seq}, we find that the change at $a \in \bR$ can be described by $\muhom(F_{\cur}, T_a R_\theta F_{\cur})|_{\{ \tau >0 \}}$, which is supported in (the conification of)
\[
    \DMS(F_{\cur}) \cap T_a\DMS(R_\theta F_{\cur}) 
    \subset \rho^{-1}((\cur \times \cur) \cap R_\theta(\cur \times \cur)).
\]
This proves the assertions (A) as well as (B) since 
\[
    \DMS(F_{\cur}) \cap T_a\DMS(R_\theta F_{\cur}) \cap \rho^{-1}(\Delta_{\cur}) 
    = \DMS(F_{\cur}) \cap T_a\DMS(F_{\cur}) \cap \rho^{-1}(\Delta_{\cur}) = \varnothing
\]
for $a \in \bR \setminus \pi \bZ$.
The most technical part is the proof of the assertion (C). 
For that purpose, we consider the value $a(\theta,\cur)$ informally defined as  
\[
    a(\theta,\cur)
    = 
    \inf \{ a \in \bR_{\ge 0} \mid \text{$v$ can be lifted to $\Hom(F_{\cur},T_{-a}R_\theta F_{\cur})$} \}, 
\]
where $v \in H^1(\End(F_{\cur})) \simeq H^1(S^1)$ is the generator. 
Then $-a(\theta,\cur)$ is a critical value, and we will prove $a(\theta,\cur) \in (0,\pi)$ for any $\theta$ in \cref{lem:spectral_invariant}. 
Most of this section is devoted to the proof of this lemma, for which we will study $\muhom(F_{\cur}, R_\theta F_{\cur})$.

\begin{remark}\label{rem:suff_cond_muhom}
    By the arguments in this section, we will find the following.
    For a fixed $\theta \in (0,\pi)$, if there exists a critical value $a_0 \in \bR$ such that 
    \[
        \Gamma(\rho^{-1}(\Delta_{\cur});\muhom(F_{\cur}, T_{a_0} R_\theta F_{\cur})|_{\rho^{-1}(\Delta_{\cur})}) \simeq 0,
    \]
    then $C$ inscribes a $\theta$-rectangle.
    In particular, to prove the existence of a $\theta$-rectangle, it is enough to show the cohomology vanishing for $a_0=-a(\theta,\cur)$.
    The only reasonable case the authors know for ensuring the vanishing is the assumption for the conic microsupport in \cref{thm:main_sheaves}. 
\end{remark}

Let us start the proof of the theorem.
First note that by \cref{prop:Jordan_Cauchy}, which will be proved in \cref{section:curves}, $F_{\cur}$ is limit constructible.
We define
\[
    \cV_{C,\theta} \coloneqq \ell^! q_* \cHom^\star(F_{\cur}, R_\theta F_{\cur}) \in \cT(\pt),
\]
where $q \colon M^2 \times \bR_t/\pi\bZ \to \bR_t/\pi\bZ$ is the projection.
This object $\cV_{C,\theta}$ is also limit constructible.
We introduce the self-map on $\bC^2$ by
\[
    R^\phi_\theta \coloneqq (\phi \times \phi)^{-1} R_\theta (\phi \times \phi).
\]
Note that $R^\phi_\pi=R_\pi$.
We also write $R^\phi_\theta$ for the GKS kernel $\cK(\phi \times \phi)^{\ostar -1} R_\theta \cK(\phi \times \phi)$ by abuse of notation.
By \cref{lem:T-linear}, we have an isomorphism in $\cT(\pt)$:
\[
    \cV_{C,\theta} \simeq \ell^! q_* \cHom^\star(\Fstd,R^\phi_\theta \Fstd).
\]
The continuation morphism $\cV_{C,0}\to \cV_{C,\theta}\to \cV_{C,\pi}$ is compatible with the continuation morphism $\cV_{\cstd,0}\to \cV_{\cstd,\pi}$. 
Since we have a homotopy between $(R_\theta)_{\theta \in [0,\pi]}$ and $(R^\phi_\theta)_{\theta \in [0,\pi]}$ relative to the boundary, we find that the continuation morphisms $\id \to R_\pi$ and $\id \to R^\phi_\pi$ are the same via the identification $R_\pi \simeq R^\phi_\pi$. 
Indeed, we get the result when $\phi$ is smooth by the argument in \cite[Subsection~3.1]{Kuowrap}, and for a Hamiltonian homeomorphism $\phi$, we obtain the result by taking limits. 

For $a,a' \in \bR$ with $a \le a'$ and $\theta, \theta' \in \bR$ with $\theta \le \theta'$, we denote the continuation morphism by 
\[
    \tau_{a,a'}^{\theta,\theta'} \colon T_a R_{\theta}^\phi F_{\cstd} \to T_{a'} R_{\theta'}^\phi F_{\cstd}.
\]
Recall that we let $v \in H^1(S^1)$ be a generator. 

\begin{lemma}\label{lem:action_zero_smalltheta}
    For any $\theta \in (0,\pi)$, the right action of $v \in H^1(S^1) \simeq H^1(\End(\Fstd))$ on the cohomology of $\muhom(\Fstd,R^\phi_\theta \Fstd)$ that corresponds to the morphism $\tau_{0,0}^{0,\theta} \colon \Fstd \to R^\phi_\theta \Fstd$ is zero. 
\end{lemma}
\begin{proof}
    Take $0= \theta_0 < \theta_1 < \theta_2 < \dots <\theta_n < \theta_{n+1}=\theta$.
    Then the canonical morphism $\muhom(\Fstd,\Fstd) \to \muhom(\Fstd,R^\phi_\theta \Fstd)$ factors as follows:
    \[
    \begin{tikzcd}[column sep = 1mm]
        \muhom(\Fstd,\Fstd) \ar[rr] \ar[rd] & & \muhom(\Fstd,R^\phi_\theta \Fstd) \\
        & \bigotimes_{i=0}^n \muhom(R^\phi_{\theta_i} \Fstd, R^\phi_{\theta_{i+1}} \Fstd). \ar[ru]
    \end{tikzcd}
    \]
    Recall that $\Lambda$ defined in \eqref{eq:Lambda} and consider its conification $\bR_{>0}\Lambda \subset T^*(M \times M \times \bR_t/\pi\bZ)$.
    Note that $\muhom(\Fstd,\Fstd) \simeq \bfk_{\bR_{>0}\Lambda}$.  
    The support of the sheaf $\bigotimes_{i=0}^n \muhom(R^\phi_{\theta_i} \Fstd, R^\phi_{\theta_{i+1}} \Fstd)$ is contained in 
    \[
        \bigcap_{i=0}^n \rho^{-1}R^\phi_{\theta_i}(\cstd \times \cstd) \cap \bR_{>0}\Lambda
    \]
    By taking refinements, we find that the canonical morphism factors through the limit as follows:
    \[
    \begin{tikzcd}[column sep = 1mm]
        \muhom(\Fstd,\Fstd) \ar[rr] \ar[rd] & & \muhom(\Fstd,R^\phi_\theta \Fstd) \\
        & \lim \bigotimes_{i=0}^n \muhom(R^\phi_{\theta_i} \Fstd, R^\phi_{\theta_{i+1}} \Fstd), \ar[ru]
    \end{tikzcd}
    \]
    where the limit in the second row is taken with respect to all the refinements.
    The support of $\lim\bigotimes_{i=0}^n \muhom(R^\phi_{\theta_i} \Fstd, R^\phi_{\theta_{i+1}} \Fstd)$ is contained in 
    \[
        \bigcap_{\theta' \in [0,\theta]} \rho^{-1} R^\phi_{\theta'}(\cstd \times \cstd) \cap \bR_{>0} \Lambda. 
    \]
    We say an arc in $\cur$ is a $\theta$-arc in $\cur$ if there exists $z_0 \in \bC$ and $r>0$ such that the arc coincides the arc with angle $\theta$ in the circle $\{ z \in \bC \mid |z-z_0|=r \}$. 
    If $(z,z')\in \bigcap_{\theta' \in [0,\theta]} R^\phi_{\theta'}(\cstd \times \cstd)\setminus \Delta_{\cstd}$, then there exist two $\theta$-arcs in $\cur$ (with counterclockwise directions) and $\phi(z), \phi(z')\in C$ are both starting points of these $\theta$-arcs.
    We set 
    \begin{equation}
        Z = \lc z \in C \relmid 
        \begin{aligned}
            \text{$z$ is a starting point of a (counterclockwise) $\theta$-arc in $\cur$}
        \end{aligned}
         \rc.
    \end{equation}
    
    We assume that $\cur$ is not a circle.
    In this case, we can take an open subset $U \subset \cstd$ such that $\phi(U)$ has no intersection with $Z$. 
    Then, we have
    \[
        \bigcap_{\theta' \in [0,\theta]} R^\phi_{\theta'}(\cstd \times \cstd)\cap ((U \times \cstd \cup \cstd \times U)\setminus \Delta_{\cstd})=\varnothing.
    \]
    Setting 
    \[
        \Xi \coloneqq \rho^{-1}(\cstd \times \cstd \setminus ((U \times \cstd \cup \cstd \times U)\setminus \Delta_{\cstd})) \cap \bR_{>0} \Lambda,
    \]
    we find that the right action of $v$ on $\Gamma(\Xi; \muhom(\Fstd, \Fstd))$ is zero since the morphism~\eqref{eq:microlocalization_morphism} maps $v$ to $v\otimes 1+1\otimes v$ and the restriction of $v\otimes 1+1\otimes v$ to $\Xi$ is zero.

    For the case that $\cur$ is a circle, this vanishing of $v\otimes 1+1\otimes v$ on the support is obvious from an explicit calculation of the support. 
    This completes the proof. 
\end{proof}

\begin{lemma}\label{lem:v_continuation}
    For any $\theta \in (0,\pi)$, the composite of the morphisms $\tau_{0,0}^{\theta,\pi} \colon R^\phi_\theta \Fstd \to R^\phi_\pi \Fstd \simeq R_\pi \Fstd$ and $R_\pi v \colon R_\pi \Fstd \to R_\pi \Fstd[1]$ is zero in $\Gamma(\{ \tau >0 \}; \muhom(R^\phi_\theta \Fstd, R_\pi \Fstd))[1]$.
    Moreover, the composite
    \begin{equation}\label{eq:muhom_vaction}
        \muhom(\Fstd,R^\phi_\theta \Fstd) \xrightarrow{\circ v} \muhom(\Fstd,R^\phi_\theta \Fstd)[1] \to \muhom(\Fstd, R_\pi \Fstd)[1]
    \end{equation}
    is the zero morphism.
\end{lemma}
\begin{proof}
    The first assertion can be proved in a similar way in \cref{lem:action_zero_smalltheta}.
    Since the morphism~\eqref{eq:muhom_vaction} is equal to     
    \[
        \muhom(\Fstd, R^\phi_\theta \Fstd) \to \muhom(\Fstd, R_\pi \Fstd) \xrightarrow{v \circ} \muhom(\Fstd, R_\pi \Fstd)[1]
    \]
    with $R_\pi v = v$, the second assertion follows.
\end{proof}

Now we consider the following commutative diagram whose rows are (co)fiber sequences by \cref{lem:muhom_fiber_seq}:
\[
\begin{tikzcd}[column sep = 3mm]
    {\displaystyle \colim_{\varepsilon \to 0} \Hom(\Fstd, T_{-\varepsilon}\Fstd)} \ar[r] \ar[d] & \End(\Fstd) \ar[r] \ar[d] & \Gamma(\{\tau >0\};\muhom(\Fstd, \Fstd)) \ar[d] \\
    {\displaystyle \colim_{\varepsilon \to 0} \Hom(\Fstd, T_{-\varepsilon} R^\phi_\theta \Fstd)} \ar[r] & \Hom(\Fstd, R^\phi_\theta \Fstd) \ar[r] & \Gamma(\{\tau >0\};\muhom(\Fstd, R^\phi_\theta \Fstd)) \mathrlap{.}
\end{tikzcd}
\]
Then the image of $v$ in the right below is zero by \cref{lem:action_zero_smalltheta}.
We take an arbitrary element $w^\theta \in \colim_{\varepsilon \to 0} \Hom(\Fstd, T_{-\varepsilon} R^\phi_\theta \Fstd)[1]$ that is mapped to $\tau_{0,0}^{0,\theta} v \in \Hom(\Fstd, R^\phi_\theta \Fstd)[1]$. 
Note that the continuation morphism $\tau^{\theta,\pi}_{-\varepsilon,-\varepsilon}$ induces a morphism 
\[
    \colim \tau^{\theta,\pi}_{-\varepsilon,-\varepsilon} 
    \colon
    \colim_{\varepsilon \to 0} \Hom(F_{\cstd}, T_{-\varepsilon}R^\phi_\theta F_{\cstd})
    \to 
    \colim_{\varepsilon \to 0} \Hom(F_{\cstd}, T_{-\varepsilon}R_\pi F_{\cstd}).
\]

\begin{lemma}\label{lem:wtheta_independence}
    The element $\colim \tau_{-\varepsilon,-\varepsilon}^{\theta,\pi} w^\theta v \in \colim_{\varepsilon \to 0} \Hom(\Fstd, T_{-\varepsilon}R_{\pi} \Fstd)[2]$ is independent of the choices of $\theta\in (0,\pi)$ and $w^\theta$.
\end{lemma}
\begin{proof}
    First, fix $\theta$ and consider two elements $w^\theta_0$ and $w^\theta_1$ that are mapped to $\tau^{0,\theta}_{0,0}v$.
    By the (co)fiber sequence above, the difference $w^\theta_0 - w^\theta_1$ is written as the image of an element $\alpha \in \Gamma(\{\tau >0\};\muhom(\Fstd, R^\phi_\theta \Fstd))$.
    The morphism that sends $\alpha$ to $\colim \tau_{-\varepsilon,-\varepsilon}^{\theta,\pi} (w^\theta_0-w^\theta_1) v$ factors the morphism 
    \[
    \Gamma(\{\tau >0\};\muhom(\Fstd, R^\phi_\theta \Fstd))\to \Gamma(\{\tau >0\};\muhom(\Fstd, R_\pi \Fstd))[1],
    \]
    which is zero by \cref{lem:v_continuation}. 
    This proves $\colim \tau_{-\varepsilon,-\varepsilon}^{\theta,\pi} (w^\theta_0-w^\theta_1) v=0$. 

    Next, we will prove the independence on $\theta$.
    Let $\theta \le \theta'$ and take $w^{\theta}$ and $w^{\theta'}$ that are mapped to $v$. 
    Then we can apply the above argument to the two element $\colim \tau_{-\varepsilon,-\varepsilon}^{\theta,\theta'}w^\theta$ and $w^{\theta'}$ in $\colim_{\varepsilon \to 0} \Hom(\Fstd, T_{-\varepsilon} R^\phi_\theta \Fstd)[1]$, which prove the lemma. 
\end{proof}

\begin{lemma}\label{lem:colim_nonzero}
    The $\colim \tau_{-\varepsilon,-\varepsilon}^{\theta,\pi} w^\theta v \in \colim_{\varepsilon \to 0} \Hom(\Fstd, T_{-\varepsilon}R_{\pi} \Fstd)[2]$ is non-zero. 
\end{lemma}
\begin{proof}
    By \cref{lem:wtheta_independence}, it is enough to show the claim for a sufficiently small $\theta>0$.

    Let us first consider the case $\phi$ is a Hamiltonian diffeomorphism with compact support. 
    In this case, we will reduce the problem to the case of the standard circle $\cstd$. 
    There exists a bi-Lipschitz constant $B$ such that 
    \[
        \frac{1}{B} d_E(z,z') \le d_E(\phi(z),\phi(z')) \le B d_E(z,z')
    \]
    for any $z,z' \in \bC$, where $d_E$ stands for the Euclidean metric.
    Note that $R_\theta$ is generated by $H(z_1,z_2)=|z_1-z_2|^2/4$ and $R^\phi_\theta$ is generated by $H^\phi=H \circ (\phi \times \phi)$, which implies 
    \[
        \frac{1}{B^2} H \le H^\phi \le B^2 H.
    \]
    Hence, as positive Hamiltonian isotopies, we have
    \[
        \id \le R_{\theta/B^2} \le R^\phi_\theta \le R_{B^2\theta} \quad \text{for $\theta \ge 0$},
    \]
    which gives continuation morphisms.
    
    We take $\theta>0$ satisfying $B^2\theta<\pi$.
    Then, for $0<\varepsilon<\theta/B^2$, we have the following interleaving
    \begin{equation}\label{eq:interleaving}
        \Hom(\Fstd, T_{-\varepsilon}R_{\theta/B^2} \Fstd) \to \Hom(\Fstd, T_{-\varepsilon}R^\phi_\theta \Fstd) \to \Hom(\Fstd, T_{-\varepsilon} R_\pi \Fstd).
    \end{equation}
    We take an element $w^\theta_\varepsilon \in \Hom(\Fstd, T_{-\varepsilon} R_{\theta/B^2} \Fstd)[1]$ that is mapped to the image of $v$ in $\Hom(\Fstd, R_{\theta/B^2} \Fstd)[1]$ via the continuation morphism. 
    Its image under the first interleaving morphism in \eqref{eq:interleaving} defines an element $w^\theta \in \colim_{\varepsilon \to 0} \Hom(F_{\cstd}, T_{-\varepsilon}R^\phi_\theta F_{\cstd})[1]$, which is mapped to $\tau_{0,0}^{0,\theta} v \in \Hom(\Fstd, R^\phi_\theta \Fstd)[1]$.
    By the arguments in \cref{subsection:standard_circle}, we find that $w^\theta_\varepsilon v \in \Hom(\Fstd, T_{-\varepsilon} R_{\theta/B^2} \Fstd)[2]$ is non-zero and 
    both of the morphisms 
    \begin{align*}
        \Hom(\Fstd, T_{-\varepsilon}R_{\theta/B^2} \Fstd) & \to \Hom(\Fstd, T_{-\varepsilon} R_\pi \Fstd) \quad  \text{and} \\
        \Hom(\Fstd, T_{-\varepsilon} R_\pi \Fstd) & \to \colim_{\varepsilon \to 0}  \Hom(\Fstd, T_{-\varepsilon} R_\pi \Fstd)
    \end{align*}
    are isomorphisms by \cref{prop:standard_structure}.
    Hence $\colim \tau_{-\varepsilon,-\varepsilon}^{\theta,\pi} w^\theta v$ is non-zero.
    
    Now we consider the continuous case and take a sequence of Hamiltonian diffeomorphisms with compact support $(\phi_n)_n$ that converges to a Hamiltonian homeomorphism $\phi$ in the $C^0$-sense. We take $(\phi_n)_n$ so that each $C_n \coloneqq \phi_n (\cstd)$ is real analytic. 
    Since the Hamiltonian function $H$ is bounded on $C\times C$, we can choose sufficiently small $\theta_0>0$ so that $\sup_{\theta\in [0,\theta_0]} d(\Fstd, R_\theta^\phi\Fstd)$ is sufficiently small. 
    Take $\varepsilon>0$ and a representative $w^\theta_\varepsilon \in \Hom(\Fstd, T_{-\varepsilon} R^\phi_{\theta_0} \Fstd)[1]$ of $w^\theta$.
    For a sufficiently large $n$, there is a $(\delta,\delta)$-interleaving for the pair $(R^\phi_{\theta_0} \Fstd, R^{\phi_n}_{\theta_0} \Fstd)$, where $\delta < \varepsilon/100$. 
    
    Let us consider the following commutative diagram:
    \[
    \begin{tikzcd}
        \Fstd \ar[r] & T_{-\varepsilon} R^\phi_{\theta_0} \Fstd[1] \ar[r] \ar[d] & T_{-\varepsilon+\delta} R^{\phi_n}_{\theta_0} \Fstd[1] \ar[r]\ar[d] &R^{\phi_n}_{\theta_0} \Fstd[1]\\
        & T_{-\varepsilon} R^\phi_\pi \Fstd[1] \ar[r] & T_{-\varepsilon+\delta} R^{\phi_n}_\pi \Fstd[1]. &
    \end{tikzcd}    
    \]
    We claim that the upper morphism $\Fstd \to R^{\phi_n}_{\theta_0} \Fstd[1]$ is $\tau_{0,0}^{0,{\theta_0}}v$. 
    We postpone the proof of this claim and first prove the assertion of the lemma. 
    By the smooth case proved above, the composite of $v$ and the morphism $\Fstd \to T_{-\varepsilon+\delta} R^{\phi_n}_\pi \Fstd[1]$ is non-zero.
    Hence, the composite of $v$ and the morphism $\Fstd \to T_{-\varepsilon} R^\phi_\pi \Fstd[1]$ is also non-zero.
    Then the result follows from the fact that 
    \[
        \Hom(\Fstd, T_{-\varepsilon} R^\phi_\pi \Fstd) \to \colim_{\varepsilon \to 0} \Hom(\Fstd, T_{-\varepsilon} R^\phi_\pi \Fstd)
    \]
    is an isomorphism.
    
    Let us prove the remaining claim by investigating the following two quantities $a(\theta,C_n)$ and $b(\theta,C_n)$ defined for $\theta \in [0,\pi)$ and $C_n$ with the property $\tau_{0,0}^{0,\theta} v\neq 0\in \Gamma_{[0,\infty)}(\bR;\cV_{C_n,\theta})[1]$:
    \begin{align*}
        a(\theta,C_n) & \coloneqq \sup \{ a \in \bR_{\ge 0} \mid \text{$\tau_{0,0}^{0,\theta}v$ is in the image of $\Gamma_{[a,\infty)}(\bR;\cV_{C_n,\theta})[1]$} \}, \\
        b(\theta,C_n) & \coloneqq \sup \lc b \in \bR_{\ge 0} \relmid 
        \begin{aligned}
            & \text{there exist $w\in \Gamma_{[b,\infty)}(\bR;\cV_{C_n,\theta})[1]$ and $t \in \bR_{\ge 0}$} \\
            & \text{such that $w$ and $\tau_{0,0}^{0,\theta}v$ coincide in $\Gamma_{[-t,\infty)}(\bR;\cV_{C_n,\theta})[1]$} \\
            & \text{as non-zero elements}
        \end{aligned}
         \rc. 
    \end{align*}
    By definition $a(\theta,C_n) \le b(\theta,C_n)$, and we already know $b(\theta_0,C_n) \ge \varepsilon-\delta$. 
    We will show $a(\theta_0,C_n)=b(\theta_0,C_n)$ and then obtain the claim with the interleaving for $(R^\phi_{\theta_0} \Fstd, R^{\phi_n}_{\theta_0} \Fstd)$. 
    We will prove it by contradiction and suppose that $a(\theta_0,C_n) \neq b(\theta_0,C_n)$.
    Consider the real number
    \[
        \theta_1\coloneqq \inf\{\theta\in [0,\theta_0]\mid a(\theta,C_n)
         \neq b(\theta,C_n)\}. 
    \]
    By the analyticity of $C_n$, the family $(\Hom (\Fstd, R_{\theta}^{\phi_n}\Fstd))_{\theta}$ is constant for sufficiently small $\theta>0$. 
    By the interleaving with $\cstd$ as above, $H^1(\Hom (\Fstd, R_{\theta}^{\phi_n}\Fstd))$ is $1$-dimensional, and hence it contains a unique non-zero element. 
    This proves $a(\theta,C_n)=b(\theta,C_n)$ for a sufficiently small~$\theta$, which implies $\theta_1>0$. 
    Consider the continuous family $(\cV_{C_n,\theta})_{\theta}$ of constructible sheaves on $\bR$, which can be regarded as a family of persistence modules (in the derived sense). 
    For $0\le \theta<\theta_1$, the element $\tau_{0,0}^{0,\theta} v$ corresponds to an interval module that is a summand of $\cV_{C_n,\theta}$ and has a length close to $\pi$. 
    When $\theta$ exceeds $\theta_1$, a change of basis occurs and the element no longer corresponds to a single interval module. 
    For such a change of basis, there needs to be another interval module of the same length. 
    However, since $\theta_0$ is sufficiently small, such an interval module cannot exist, which makes a contradiction. 
\end{proof}

\begin{lemma}\label{lem:non_zero_alltheta}
    For any $\theta \in (0,\pi)$, the element $\tau_{0,0}^{0,\theta} v$ is non-zero in $\Hom(\Fstd, R^\phi_\theta \Fstd)[1] \simeq \Hom(F_{\cur}, R_\theta F_{\cur})[1]$.
\end{lemma}
\begin{proof}
    If $\tau_{0,0}^{0,\theta} v=0$, we can take $w^\theta$ as the zero element. 
    This contradicts to \cref{lem:wtheta_independence,lem:colim_nonzero}.
\end{proof}

By \cref{lem:non_zero_alltheta,cor:pi_torsion}, we can define $a(\theta,C)$ by
\[
    a(\theta,C) \coloneqq \sup \{ a \in \bR_{\ge 0} \mid \text{$\tau_{0,0}^{0,\theta} v$ is in the image of $\Gamma_{[a,\infty)}(\bR;\cV_{C,\theta})[1]$} \} \in \bR_{\ge 0}, 
\]
which already appeared in the proof of \cref{lem:colim_nonzero}.

\begin{lemma}\label{lem:spectral_invariant}
    For any $\theta\in (0,\pi)$, one has $a(\theta,C) \in (0, \pi)$.
\end{lemma}
\begin{proof}
    By the argument before \cref{lem:wtheta_independence}, the element $\tau_{0,0}^{0,\theta} v$ comes from $\Gamma_{[\varepsilon_\theta,\infty)}(\bR;\cV_{C,\theta})[1]$ for some $\varepsilon_\theta>0$, which shows $a(\theta,C)>0$.

    By \cref{prop:Jordan_Cauchy} in the next section, the object $\cV_{C,\theta} \in \cT(\pt)$ is limit constructible. 
    By \cref{prop:barcode_decomposition}, we find that $v$ is non-zero in $\Gamma_{[-\varepsilon,\infty)}(\bR;\cV_{C,\theta})[1]$ for a sufficiently small $\varepsilon>0$.
    By \cref{cor:pi_torsion}, $v$ does not come from $\Gamma_{[\pi-\varepsilon,\infty)}(\bR;\cV_{C,\theta})[1]$, which proves $a(\theta,C)<\pi$.
\end{proof}

We will finish the proof of \cref{thm:main_sheaves}.
The object $\cV_{C,\theta}$ has a non-zero microstalk over $a(\theta,C)$, which implies $\DMS(F_{\cur})\cap T_{-a(\theta,C)} \DMS(R_\theta F_{\cur})\neq \varnothing$.
By the assumption and \cref{lem:spectral_invariant}, we find that $\DMS(F_{\cur})\cap T_{-a(\theta,C)} \DMS(F_{\cur})=\varnothing$. 
Thus, we have $(\DMS(F_{\cur})\cap T_{-a(\theta,C)} \DMS(R_\theta F_{\cur})) \setminus \rho^{-1}(\Delta_C) \neq \varnothing$, which corresponds to $\theta$-rectangles on $\cur$. 
This completes the proof of \cref{thm:main_sheaves}.

\section{Jordan curves}\label{section:curves}

In this section, we deduce \cref{thm:main_intro} from \cref{thm:main_sheaves}. 
We also deduce \cref{cor:intro_rectifiable,cor:intro_locallymonotone} from \cref{thm:main_intro}. 
Throughout this section, we let $\bD_q$ be the open disk $\{z\in \bC\mid |z|<q \}$ in $\bC \simeq \bR^2$ for $q>0$.
We also set $\bA_q\coloneqq \{z\in \bC \mid  q<|z|<1\}$ for $q \in (0,1)$.
For a Jordan curve $\cur$, we let $A(\cur)$ denote the area of the open domain bounded by $\cur$.

\subsection{Proof of the main theorem}\label{subsec:proof_main_theorem}

For a proof of \cref{thm:main_intro}, we first prove the following:

\begin{proposition}\label{prop:Jordan_Cauchy}
    Let $(c_n \colon S^1 \to \bR^2)_n$ be a sequence of smooth curves.
    Assume that $(c_n)_n$ converges to a Jordan curve $c$ in the $C^0$-sense and the area of the domain bounded by $C_n=c_n(S^1)$ and $C=c(S^1)$ are $\pi$, that is, $A(C_n)=\pi$ and $A(C)=\pi$. 
    Then the sequence of sheaf quantizations $(F_{C_n})_n$ is a Cauchy sequence (after translated to the $\bR_t/\pi \bZ$-direction), whose limit object $F$ is limit constructible.  

    Moreover, if there exists a Hamiltonian homeomorphism with compact support $\phi$ such that $C=\phi(S^1)$, then $F \simeq F_{\cur} \coloneqq \cK(\phi \times \phi)F_{\cstd}$.
\end{proposition}
\begin{proof}
    We may assume that the origin is bounded by $C_n$ for all $n$.
    
    \noindent (a) First we will prove $(F_{C_n})_n$ is a Cauchy sequence.

    Let $D$ be the open domain bounded by $\cur$.
    We take a biholomorphism $\psi \colon \bD_1 \to D$ with $\psi(0)=0$ and extend it to a homeomorphism $\overline{\psi} \colon \overline{\bD_1} \to \overline{D}$ by the Riemann mapping theorem and the Carath\'eodory theorem. 
    There is a strictly increasing function $g \colon (0,1] \to (0,1]$ such that the area $D_a \coloneqq \psi(\bD_{g(a)})$ is $\pi a^2$. 
    Then, the family of open subdomains $(D_a)_{a \in (0,1]}$ satisfy the following:
    \begin{itemize}
        \item if $a < a'$, then $\overline{D_a} \subset D_{a'}$; 
        \item for each $a\in (0,1)$, the boundary $\partial D_a$ is a smooth Jordan curve;
        \item there exists a positive real number $L$ such that 
        \[
            \frac{1}{2\pi} \lb \max_{\{a\} \times [0,2\pi]} \widetilde{\theta_\psi} -\min_{\{a\} \times [0,2\pi]} \widetilde{\theta_\psi} \rb \le L_\psi
        \]    
        for any $a \in (0,1]$.
        Here $\widetilde{\theta_\psi} \colon (0,1] \times [0,2\pi] \to \bR$ denotes a lift of 
        \[
            (0,1] \times [0,2\pi] \xrightarrow{(r,\theta) \mapsto r e^{\sqrt{-1}\theta}} \bD \setminus \{0\} \xrightarrow{\overline{\psi}} D\setminus\{0\} \xrightarrow{\theta} \bR/2\pi \bZ,
        \]
        where $\theta$ denotes the locally defined argument.
        This $L_\psi$ depends only on $\psi$. 
    \end{itemize}
    To prove the last claim, we take $\varepsilon >0$ and consider the annulus $\bA_\varepsilon = \{z\in \bC \mid  \varepsilon<|z|<1\}$.
    Then we can apply \cref{lem:argument_osc} below to get 
    \[
        \frac{1}{2\pi} \lb \max_{\{a\} \times [0,2\pi]} \widetilde{\theta_\psi} -\min_{\{a\} \times [0,2\pi]} \widetilde{\theta_\psi} \rb \le L,
    \]
    where $L$ can be chosen so that
    \[
        L \le \frac{1}{2\pi} \max \lc \max_{\{\varepsilon\} \times [0,2\pi]} \widetilde{\theta_\psi} -\min_{\{\varepsilon\} \times [0,2\pi]} \widetilde{\theta_\psi}, 
        \max_{\{1\} \times [0,2\pi]} \widetilde{\theta_\psi} -\min_{\{1\} \times [0,2\pi]} \widetilde{\theta_\psi} 
        \rc +1.
    \]
    Since $\psi$ is differentiable at $0$, given $\delta>0$, there exists a sufficiently small $\varepsilon>0$ such that $\max_{\{a\} \times [0,2\pi]} \widetilde{\theta_\psi} -\min_{\{a\} \times [0,2\pi]} \widetilde{\theta_\psi} \le 2\pi+\delta$ for any $a \in (0,\varepsilon]$. 
    It suffices to define
    \[
        L_\psi \coloneqq \frac{1}{2\pi} \max \lc 2\pi, 
        \max_{\{1\} \times [0,2\pi]} \widetilde{\theta_\psi} -\min_{\{1\} \times [0,2\pi]} \widetilde{\theta_\psi} 
        \rc +1,
    \]
    which proves the claim.    
    
    Take $a<1$ that is sufficiently close to $1$. 
    By the $C^0$-convergence, there exists $N$ such that if $n \ge N$ then $C_n$ is included in the complement of $\overline{D_a}$.
    Let $A_{a,n}$ be the domain between $\partial D_a$ and $C_n$.
    Note that the area of $A_{a,n}$ is $\pi(1-a^2)$.     
    There exist a unique real number $q\in (0,1)$ such that the standard annulus $\bA_q = \{z\in \bC \mid  q<|z|<1\}$ is biholomorphic to the open domain $A_{a,n}$. 
    Take a biholomorphism $\varphi_n \colon \bA_{q} \to A_{a,n}$ so that the continuous extension $\overline{\varphi_n} \colon \overline{\bA_{q}} \to \overline{A_{a,n}}$ of $\varphi_n$ satisfies
    \begin{itemize}
        \item $\overline{\varphi_n}$ sends $\partial \bD_1$ to $C_n$;
        \item $\overline{\varphi_n}$ sends $q \in \partial \bD_q$ to $\psi(g(a)) \in \partial D_a$, where $g(a)$ is regarded as a point on $\partial \bD_{g(a)}$.
    \end{itemize}
    Since the boundary components of $A_{a,n}$ are smooth curves, $\overline{\varphi_n}$ is smooth also at the boundaries by \cite[Chapter~II. Cor.~4.6]{GM05harmonic}. 
    Let $\widetilde{\theta_n} \colon (q,1)\times [0,2\pi] \to \bR$ be a lift of $(q,1)\times [0,2\pi] \to \bA_q \xrightarrow{\varphi_n} A_{a,n} \xrightarrow{\theta} \bR/2\pi \bZ$.   
    By the condition of $\overline{\varphi_n}$, we have 
    \[
        \max_{\{q\} \times [0,2\pi]} \widetilde{\theta_n} -\min_{\{q\} \times [0,2\pi]} \widetilde{\theta_n} =  \max_{\{q\} \times [0,2\pi]} \widetilde{\theta_\psi} -\min_{\{q\} \times [0,2\pi]} \widetilde{\theta_\psi}.
    \]
    Since $(c_n)_n$ converges $c$ in the $C^0$-sense, there exists a sequence of self-homeomorphisms $(\sigma_n)_n$ of $\partial \bD_1$ such that $\overline{\varphi_n} \circ \sigma_n$ converges to $\overline{\psi}|_{\partial \bD_1}$ in the $C^0$-sense. 
    Take a lift $\widetilde{\theta_n'}$ of $[0,2\pi] \to \partial \bD_1 \xrightarrow{\overline{\varphi_n} \circ \sigma_n} \overline{\varphi_n}(\partial \bD_1) \xrightarrow{\theta} \bR/2\pi \bZ$.  
    We will prove the inequality
    \begin{equation}\label{eq:argument_ineq}
        \left| \left( \max_{\{1\} \times [0,2\pi]} \widetilde{\theta_n} -\min_{\{1\} \times [0,2\pi]} \widetilde{\theta_n} \right) - \left( \max_{[0,2\pi]} \widetilde{\theta_n'} -\min_{[0,2\pi]} \widetilde{\theta_n'} \right) \right| \le 2\pi.
    \end{equation}
    By abuse of notation, we also write $\widetilde{\theta_n}$ for a lift of $\bR \xrightarrow{\theta \mapsto e^{\sqrt{-1}\theta}} \partial \bD_1 \to \overline{\varphi_n}(\partial \bD_1) \to \bR/2\pi \bZ$ to $\bR \to \bR$. 
    Then, we get
    \[
        \max_{\{1\} \times [0,2\pi]} \widetilde{\theta_n} -\min_{\{1\} \times [0,2\pi]} \widetilde{\theta_n}  + 2\pi = \max_{\{1\} \times [0,4\pi]} \widetilde{\theta_n} -\min_{\{1\} \times [0,4\pi]} \widetilde{\theta_n}.
    \]
    Moreover, there exists $b \in [0,2\pi]$ satisfying
    \[ 
        \max_{[0,2\pi]} \widetilde{\theta_n'} -\min_{[0,2\pi]} \widetilde{\theta_n'} = \max_{\{1\} \times [b,b+2\pi]} \widetilde{\theta_n} -\min_{\{1\} \times [b,b+2\pi]} \widetilde{\theta_n},
    \]
    which proves the inequality~\eqref{eq:argument_ineq}.       
    By \eqref{eq:argument_ineq} and the $C^0$-convergence, for a sufficiently large $n$, we have
    \[
        \left| \left( \max_{\{1\} \times [0,2\pi]} \widetilde{\theta_n} -\min_{\{1\} \times [0,2\pi]} \widetilde{\theta_n} \right) - \left( \max_{\{1\} \times [0,2\pi]} \widetilde{\theta_\psi} -\min_{\{1\} \times [0,2\pi]} \widetilde{\theta_\psi} \right) \right| \le 2.1\pi.
    \]
    Thus, setting $L' \coloneqq L_\psi + 2.1/2$, by \cref{lem:argument_osc}, we have
    \[
        \frac{1}{2\pi} \lb \max_{\{u\} \times [0,2\pi]} \widetilde{\theta_n} -\min_{\{u\} \times [0,2\pi]} \widetilde{\theta_n} \rb \le L'
    \]
    for any $u \in (q,1)$. 

    Let $\partial D_a'$ be the curve $\partial D_a$ rescaled by the flow $\phi^{d\theta}$ defined below so that $A(\partial D_a')=\pi$.
    For $u \in (q,1)$, put $C_u \coloneqq \varphi (\partial \bD_u)$ and let $ C_u'$ be the curve rescaled by the flow $\phi^{d\theta}$ so that $A(C_u')=\pi$.
    By \cref{lem:annulus_energy} below, for a sequence $(a_i)_i$ of real numbers in $(q, 1)$ converging to $q$ from above, the sequence of constructible sheaves $(F_{C_{a_i}'})_i$ is Cauchy.  
    
    We see that the limit object $F'$ of $(F_{C_{a_i}'})_i$ is isomorphic to $F_{\partial D_a'}$ as follows. 
    By the microsupport estimate for the limit object, the microsupport of $F'$ coincides with that of $F_{\partial D_a'}$ since $\overline{\varphi_n}$ is smooth also at the boundaries. 
    By taking a compactly supported Hamiltonian diffeomorphism sending $\partial D_a'$ to $\cstd$ and applying the corresponding GKS kernel to $F_{\partial D_a'}$ and $F'$, the assertion $F_{\partial D_a'} \simeq F'$ follows from \cref{lem:Fcstdunique}. 
    Similarly, for a sequence $(a_i)_i$ of real numbers in $(q, 1)$ converging to $1$ from below, the sequence of constructible sheaves $(F_{C_{a_i}'})_i$ is Cauchy and converges to $F_{C_n}$. 
    
    Again by \cref{lem:annulus_energy}, for any $q< u_0 < u_1<1$, 
    \[
        d(F_{C_{u_0}'}, F_{C_{u_1}'}) \le 2(L'+1)(A(C_{u_1})-A(C_{u_0})) \le 2(L'+1)\pi(1-a^2).
    \]
    By tanking limits, we obtain
    \[
        d(F_{\partial D_a'}, F_{C_n}) \le 2(L'+1)\pi(1-a^2).
    \]
    Hence, for $m,n \ge N$, we have 
    \[
        d(F_{C_n}, F_{C_m}) \le 4(L'+1)\pi(1-a^2),
    \]
    which proves that $(F_{C_n})_n$ is a Cauchy sequence.    
    Since each $F_{C_n}$ is limit constructible, a limit object $F$ is also limit constructible.
    \smallskip

    \noindent (b) Let us prove the second assertion and suppose that $\cur=\phi(S^1)$ for some Hamiltonian homeomorphism with compact support $\phi$. 
    Then there exists a sequence of Hamiltonian diffeomorphisms $(\phi_n)_n$ that converges to $\phi$ in the $C^0$-sense. 
    The sequence $(\cK(\phi_n \times \phi_n)F_{\cstd})_n$ is a Cauchy sequence, and its limit object is $\cK(\phi \times \phi)F_{\cstd}$ by definition. 
    Then the sequence $(F_k)_k$ with 
    \[
        F_k = 
        \begin{cases}
            F_{C_n} & (k=2n-1), \\
            \cK(\phi_n \times \phi_n) F_{\cstd} & (k=2n)
        \end{cases}
    \]
    is also a Cauchy sequence. 
    Since each pair of the limit objects of the three sequences $(F_{C_n})_n$, $(\cK(\phi_n \times \phi_n)F_{\cstd})_n$, and $(F_k)_k$ has distance zero, we conclude that $F \simeq \cK(\phi \times \phi)F_{\cstd}$ by the limit constructibility and \cref{prop:limit_constructible}. 
\end{proof}

We fix some notation. 
Let $g$ be the standard metric on $\bC$ and set $\omega \coloneqq d\lambda=d\xi \wedge dx$ be the symplectic form on $\bC \simeq T^*\bR_x$. 
We have $\omega(X, Y) = g(X, \sqrt{-1}Y)$. 
Let $r, \theta \colon \bC \setminus \{0\} \to \bR$ be the radius and the (locally defined) argument. 
We remark that $d\theta(X) = - \frac 1r dr(\sqrt{-1}X)$, for all $X$.
For a smooth function $f$ (locally defined) on $\bC$, we let $\nabla_f$ be the gradient vector field with respect to $g$ and $X_f$ the Hamiltonian vector field. 
For a $1$-form $\alpha$ (locally defined) on $\bC$, we let $X_\alpha$ be the symplectic vector field with respect to $\omega$. 
We have $g(\nabla_f, X) = df(X)$, $\omega(X_\alpha, X) = -\alpha(X)$, for all $X$. 
We write $\phi^\alpha$ for the symplectic isotopy generated by $X_\alpha$.
We obtain 
\[
    \omega(X_{d\theta}, X) = -d\theta(X) =  \frac 1r dr(\sqrt{-1}X) =  \frac 1r g(\nabla_r, \sqrt{-1}X)
    = \frac 1r \omega(\nabla_r, X)
\]
and thus $X_{d\theta} = \frac 1r \nabla_r$. 
We deduce an expression of the symplectic isotopy $\phi^{d\theta}_{s}$ in the coordinates $(r,\theta)$:
\[
    \phi^{d\theta}_{s}(r, \theta) = (\sqrt{2s + r^2}, \theta).
\]

\begin{lemma}\label{lem:argument_osc}
    Let $\varphi \colon \bA_{q} \to \bC$ be a biholomorphism onto its image $A$. Assume that $\varphi$ admits a continuous extension $\overline{\varphi} \colon \overline{\bA_{q}} \to \overline{A}$ and $0\notin \overline A$. 
    Let $\tilde{\theta} \colon [q,1] \times [0,2\pi] \to \bR$ be a lift of $[q,1] \times [0,2\pi] \xrightarrow{\overline{\varphi}} \bA_q\to A \xrightarrow{\theta} \bR/2\pi \bZ$. 
    Then, there exists a positive real number $L \in \bR_{>0}$ such that 
    \[
        \frac{1}{2\pi} \lb \max_{\{u\} \times [0,2\pi]} \tilde{\theta} -\min_{\{u\} \times [0,2\pi]} \tilde{\theta} \rb \le L
    \]
    for any $u \in [q,1]$.
    This $L$ can be chosen so that 
    \[
        L \le \frac{1}{2\pi} \max \lc \max_{\{q\} \times [0,2\pi]} \tilde{\theta} -\min_{\{q\} \times [0,2\pi]} \tilde{\theta}, 
        \max_{\{1\} \times [0,2\pi]} \tilde{\theta} -\min_{\{1\} \times [0,2\pi]} \tilde{\theta} 
        \rc +1.
    \]
\end{lemma}
\begin{proof}
    By abuse of notation, we write $\theta$ for $(q,1) \times [0,2\pi] \to \bA_q \xrightarrow{\theta} \bR$.
    Let $\theta' \colon (q,1) \times [0,2\pi] \to \bR$ denote the second projection.
    Then the function $\tilde{\theta} -\theta'$ defines a harmonic function on $\bA_q$. 
    Let $I_u \subset \bR$ be the image of $\partial \bD_u$ under $\tilde{\theta} -\theta'$. 
    We may assume $I_q \subset I_1$ or $I_1 \subset I_q$ by adding a harmonic function of the form $c \log r \ (c \in \bR)$ if necessary. 
    Note that this does not change the length of each $I_u$.

    By the maximum principal, $I_u$ is contained in $I_q \cup I_1$. 
    Since the values of $\theta'$ is contained in $[0,2\pi]$, the oscillation is less than or equal to $\max\{|I_q|,|I_1|\}+2\pi$, where $|I|$ denotes the length of a interval $I\subset \bR$.
\end{proof}

The essential part of the proof of the following lemma is due to St\'{e}phane Guillermou.

\begin{lemma}\label{lem:annulus_energy}
    Let $\varphi \colon \bA_{q} \to \bC$ be a biholomorphism onto its image $A$ and let $L$ be a positive real number satisfying the inequality in \cref{lem:argument_osc}. 
    For $u \in (q,1)$, set $C_u \coloneqq \varphi(\partial \bD_u)$ and assume  $A(C_{u})\le \pi$ for all $u \in (q,1)$.
    Define $C'_u$ to be the curve rescaled by $\phi^{d\theta}$ defined above such that $A(C'_u)=\pi$.
    Then, for $q < u_0 < u_1 < 1$, one has
    \[
        d(F_{C'_{u_0}},F_{C'_{u_1}}) \le 2(L+1)(A(C_{u_1})-A(C_{u_0}))
    \]
    after translating $F_{C'_{u_0}}$ by some constant to the $\bR_t/\pi \bZ$-direction.
\end{lemma}
\begin{proof}
    We may assume that $0 \in \bC$ is contained in the open domain bounded by $C_u$ for all $u \in (q,1)$. 
    We set $r' = r \circ \varphi^{-1}$, $\theta' = \theta \circ \varphi^{-1} \colon A \to \bR_x$.  
    Hence $C_u = r'^{-1}(u)$.
    Since $\varphi$ is biholomorphic, we obtain $X_{d\theta'} = \frac 1{r'} \nabla_{r'}$.

    In the following steps from (a) to (d), we will construct a Hamiltonian diffeomorphism that sends $C'_{u_0}$ to $C'_{u_1}$ and estimate the distance $d(F_{C'_{u_0}},F_{C'_{u_1}})$ with the Hamiltonian diffeomorphism.

    \noindent (a) First we will define a time-dependent closed $1$-form $\alpha=(\alpha(s))_{s \in [0,u_1-u_0]}$ on $A$ such that the flow of its symplectic vector field $\phi^{\alpha}$ satisfies $\phi^{\alpha}_s(C_{u_0}) = C_{{u_0}+s}$ for $s \in [0,u_1-u_0]$. 
    This condition is satisfied if $dr'(X_{\alpha(s)})= 1$ on $C_{{u_0}+s}$ for $s \in [0,u_1-u_0]$.
    We define a function $k$ that depends only on $s$ and $\theta'$ by 
    \[
        k(s,\theta') \coloneqq \frac{u_0+s}{\| dr' \|^2},
    \]
    where $\|dr'\|^2$ is a time-dependent function on $A$ that maps $(r'_1,\theta'_1)$ to the value of $\|dr'\|^2$ at $(u_0+s, \theta'_1)$.
    We define
    \[
        \alpha(s) \coloneqq k(s,\theta') d\theta'.
    \]
    Then, on $C_{u_0+s}$ we have
    \[
        X_{\alpha(s)} = k(s,\theta') X_{d\theta'},
    \]
    which implies $dr'(X_{\alpha(s)})= 1$. 
    Moreover, we have $d\theta'(X_{\alpha(s)})=0$ by construction. 

    \noindent (b) 
    Next, we will describe the rescaled curve $C'_u$ more precisely. 
    We have seen that $\phi^{d\theta}_{s}(\partial \bD_u) = \partial \bD_{\sqrt{2s + u^2}}$.
    Hence $A( \phi^{d\theta}_{s}(\partial \bD_u) ) = A(\partial \bD_u) + 2\pi s$.  
    Now, for a general Jordan curve $\cur$ containing $0$ in its interior domain and $\varepsilon>0$ small, $\phi^{d\theta}_{s}$ is defined and symplectic outside $\overline{\bD_\varepsilon}$.
    Hence we deduce the general equality
    \[
        A(\phi^{d\theta}_{s}(C)) = A(C) + 2\pi s.
    \]
    Thus, we can write 
    \[
        C'_u = \phi^{d\theta}_{T(u)}(C_u) \quad \text{with} \quad T(u) \coloneqq \frac{1}{2\pi}(\pi - A(C_u)).
    \]

    \noindent (c) 
    We will construct a Hamiltonian diffeomorphism that sends $C'_{u_0}$ to $C'_{u_1}$.
    We define a symplectomorphism $\psi \coloneqq \phi^{d\theta}_{T(u_0)}$ and a time-dependent closed $1$-form $\beta$ by $\beta(s) \coloneqq (\psi^{-1})^*\alpha(s)$.
    We set $a(s) = A(C_s)$ and define time-dependent function and 1-form 
    \[
        b(s) \coloneqq -\frac{1}{2\pi} \frac{da}{ds}(u_0+s), \quad d\Theta(s) = b(s) d\theta \quad (s \in [0,u_1-u_0]).
    \]
    Since
    \[
        \int_0^s b(s') \, ds' = \frac{1}{2\pi} (a(u_0)-a(u_0+s)) = T(u_0+s)-T(u_0),
    \]
    we obtain $\phi^{d\Theta}_{s} = \phi^{d\theta}_{T(u_0+s)-T(u_0)}$. 
    For $s \in [0,u_1-u_0]$, we define 
    \[
        (d\Theta \sharp \beta)(s)
        \coloneqq 
        d\Theta(s) + ((\phi^{d\Theta}_s)^{-1})^* \beta(s)
        = d\Theta(s) + ((\phi^{d\theta}_{T(u_0+s)-T(u_0)})^{-1})^*\alpha(s),
    \]
    which is a locally defined time-dependent closed $1$-form.
    We find that 
    \begin{align*}
        \phi^{d\Theta \sharp \beta}_{s} 
        & = 
        \phi^{d\Theta}_{s} \circ \phi^{\beta}_{s} \\
        & = \phi^{d\theta}_{T(u_0+s) - T(u_0)} \circ \psi \circ \phi^{\alpha}_{s} \circ \psi^{-1} \\
        & = \phi^{d\theta}_{T(u_0+s)} \circ \phi^{\alpha}_{s} \circ (\phi^{d\theta}_{T(u_0)})^{-1}, 
    \end{align*}
    which sends $C'_{u_0}$ to $C'_{u_0+s}$.
    The exactness of a locally defined closed $1$-form is determined by the integrations along closed curves that generate the first homology group of the domain.  
    Since $A(C'_{u_0+s})=A(C'_{u_0})$, the integration of $(d\Theta \sharp \beta)(s)$ along $C'_{u_0+s}$ is zero. 
    Thus $d\Theta \sharp \beta$ is a time-dependent locally defined exact $1$-form, which can be written as $dh_1$.
    This proves that $\phi^{d\Theta \sharp \beta}_{u_1-u_0}$ is the Hamiltonian diffeomorphism $\phi^{h_1}_{u_1-u_0}$ that sends $C'_{u_0}$ to $C'_{u_1}$.

    \noindent (d) Finally, we will estimate the Hofer norm of $\phi^{h_1}_{u_1-u_0}$.
    We take a smooth cut-off function on $\bC$ and extend $h_1$ to $\bC$ with the cut-off function. 

    For any $z_1,z_2 \in C_{u_0+s}'$, we take a path in $C_{u_0+s}'$ connecting these two points that does not pass $\theta'=0$.
    Then, by integrating $d\Theta \sharp \beta$ along the path, we get
    \begin{align*}
        h_1(s,z_1) - h_1(s,z_2) 
        & \le \frac{1}{2\pi} |b(s)| \lb \max_{\{u_0+s\} \times [0,2\pi]} \tilde{\theta} -\min_{\{u_0+s\} \times [0,2\pi]} \tilde{\theta} \rb 
        +
        \int_{\theta'_1}^{\theta'_2} k(s, \theta') \, d\theta',
    \end{align*}
    where $(u_0+s,\theta_i')$ in the coordinates $(r',\theta')$ corresponds to the point $z_i$ for $i=1,2$. 
    The area bounded by the arcs $\theta'$ is constant or $r'$ is constant joining the points $(u_0, \theta'_i)$, $(u_0+s, \theta'_i)$ for $i=1,2$ is written as  
    \[
        B(s, \theta'_1, \theta'_2) = \int_{r'=u_0}^{r'=u_0+s} \int_{\theta' = \theta'_1}^{\theta' = \theta'_2}
    \omega( \partial_{r'}, \partial_{\theta'}) \, d\theta' dr'.
    \]
    By using $\omega( \partial_{r'}, \partial_{\theta'}) = \omega \lb k X_{d\theta'}, \partial_{\theta'}\rb = k$, we have
    \[
    \frac{\partial B}{\partial s}(s,\theta_1',\theta'_2)
    = \int_{\theta' = \theta'_1}^{\theta' = \theta'_2} \omega( \partial_{r'}, \partial_{\theta'}) \, d\theta'
    = \int_{\theta' = \theta'_1}^{\theta' = \theta'_2} k(s,\theta') \, d\theta'.
    \]
    Since $k(s,\theta') \geq 0$ and $B(s, 0, 2\pi) = a(s) - a(u_0)$, we obtain the bound
    \[
    \int_{\theta'_1}^{\theta'_2} k(s, \theta') \, d\theta' \leq \frac{da}{ds}(s) \quad\text{for any $s$ and $\theta'_1,\theta'_2$.}
    \]
    Combining this inequality with \cref{lem:argument_osc}, we have     
    \[
        h_1(s,z_1) - h_1(s,z_2) \le (L+1) \frac{da}{ds}(u_0+s)
    \]
    Hence, we obtain
    \[
        \int_0^{u_1-u_0} \lb \max_{C'_{u_0+s}} h_1(s) - \min_{C'_{u_0+s}} h_1(s) \rb \, ds \le (L+1) (a(u_1)-a(u_0)).  
    \]
    The bound is equal to $(L+1)(A(C_{u_1})-A(C_{u_0}))$.
    \smallskip

    We will finish the proof of the lemma.
    The time-depending function $(p,p') \mapsto h_1(p,s)+h_1(p',s)$ on $\bC \times \bC$ generates a flow that sends $C'_{u_0} \times C'_{u_0}$ to $C'_{u_1} \times C'_{u_1}$ at time $s=u_1-u_0$.
    Hence, by \cite[Thm.~A.2]{AI24completeness}, there exists $c \in \bR$ such that 
    \begin{align*}
        d(F_{C'_{u_0}}, T_c F_{C'_{u_1}})
        & \le 2 \int_0^{u_1-u_0} \lb \max_{C'_{u_0+s}} h_1(s) - \min_{C'_{u_0+s}} h_1(s) \rb \, ds \\
        & \le 2(L+1)(A(C_{u_1})-A(C_{u_0})).
    \end{align*}
    This completes the proof.
\end{proof}

\begin{remark}
    Note that we can define a sheaf quantization $F_{\cur}$ for any Jordan curve $\cur$ by \cref{prop:Jordan_Cauchy}. 
\end{remark}

\begin{remark}\label{rem:positive_measure}
    Note that there are Jordan curves whose images have positive measure \cite{Lebesgue03, Osgood03}. See also \cite{NV22}.
    If the measure of $\cur$ is non-zero, $\cur$ inscribes a $\theta$-rectangle for any $\theta \in (0,\pi)$ by Lebesgue's density theorem. 
\end{remark}

Now we prove \cref{thm:main_intro}.

\begin{proof}[Proof of \cref{thm:main_intro}]
    We may assume that the measure of $\cur$ is zero by \cref{rem:positive_measure}.
    By scaling, we may also assume that the area of the open domain bounded by $C$ is $\pi$, that is, $A(C)=\pi$.
    Let $(c_n)_n$ be a sequence of smooth Jordan curves that satisfies the conditions in \cref{def:continuous_legendrian_lift}.
    Let $B_n \coloneqq A(C_n)$ be the area of the open domain bounded by $C_n$.
    Since $B_n \to \pi$ as $n \to \infty$, by scaling $C_n$ by a factor of $\sqrt{\pi/B_n}$ with respect to the origin, we may assume $B_n=\pi$ while keeping $(c_n)_n$ converges to $c$.
    By the first part of \cref{prop:Jordan_Cauchy}, the sequence of sheaf quantizations $(F_{C_n})_n$ is a Cauchy sequence, which defines a limit object $F$. 
    Combining the condition~(2) in \cref{def:continuous_legendrian_lift} with \cref{prop:limit_ss}, we find that $T_a\CMS (F)\cap \CMS(F)=\varnothing$ for $a \in \bR \setminus \pi \bZ$.

    Since the measure of $\cur$ is zero, we can construct a Hamiltonian homeomorphism with compact support $\phi$ on $T^*\bR$ such that $\cur=\phi(\cstd)$. 
    Note that the set of compactly supported Hamiltonian homeomorphism coincides with the set of compactly supported area-preserving homeomorphisms, whose proof can be found in \cite{Oh06, Sikorav07}. 
    Such a compactly supported area-preserving homeomorphisms exists by theorems by Sch\"{o}nflies and Oxtoby--Ulam~\cite{OU41}. 
    Then, by the second part of \cref{prop:Jordan_Cauchy}, we have $F \simeq F_{\cur} \coloneqq \cK(\phi \times \phi)F_{\cstd}$. 

    Hence, the result follows from \cref{thm:main_sheaves}.
\end{proof}

\begin{remark}
    Let  $c \colon S^1 \to \bR^2$ be a Jordan curve that admits a continuous Legendrian lift. 
    Assume that the image $\cur$ of $c$ has measure zero.
    Then a ``primitive'' for $c$, a function $f$ in \cref{def:continuous_legendrian_lift}, is unique up to shift regardless of how an approximating sequence $(c_n)_n$ is chosen. 
    This uniqueness follows from the fact that the sheaf quantization $F_{\cur}$ is unique up to shift and the primitive can be recovered from its conic microsupport. 
\end{remark}

\begin{remark}
    As mentioned in \cref{rem:positive_measure}, a Jordan curve with positive measure inscribes a $\theta$-rectangle for any $\theta \in (0,\pi)$.
    Thus, the Rectangular Peg Problem for any Jordan curve would be solved affirmatively if the cohomology vanishing in \cref{rem:suff_cond_muhom} for Jordan curves with measure zero.
\end{remark}

\subsection{Rectifiable curves}

Now we give an affirmative answer to the rectangle peg problem for rectifiable curves.

\begin{proposition}
    A rectifiable Jordan curve $\cur$ admits a continuous Legendrian lift. 
\end{proposition}
\begin{proof}
    Let $D$ be the open domain bounded by $\cur$.
    By the Riemann mapping theorem and the Carath\'eodory theorem, we can construct a homeomorphism $\overline{\varphi} \colon \overline{\bD_1} \to \overline{D}$ whose restriction to $\bD_1$ is a holomorphic mapping. 
    For $n \in \bZ_{\ge 2}$, we define a smooth Jordan curve $c_n \coloneqq \overline{\varphi}|_{\partial \bD_{1-1/n}}$. 
    By the Riesz--Privalov theorem, a precise form of the Riemann mapping theorem for a domain with rectifiable boundary~\cite[Thm.~6.8]{Pommerenke92}, we find that the lengths of $c_n$ converge to the length of $c$.
    Then, by the lemmas for proving Green's theorem for rectifiable curves~\cite[{{10--14}}]{Apostol57}\footnote{Note that this discussion is only written in the first edition and has been removed from the second edition onward. An overview of the discussion can also be found on Wikipedia~\cite{Wiki_Green}.}, we find that the sequence of smooth Jordan curve $(c_n)_n$ satisfies the conditions in \cref{def:continuous_legendrian_lift} .
\end{proof}

\begin{corollary}\label{cor:rectifiable_peg}
    Every rectifiable Jordan curve inscribes a $\theta$-rectangle for any $\theta \in (0,\pi)$.
\end{corollary}

\subsection{Locally monotone curves}

Stromquist~\cite{Stromquist} proved the existence of an inscribed square for a large class of Jordan curves, which he called locally monotone. 
We will also extend his result with the use of \cref{thm:main_intro}. 

Let us first recall the definition of locally monotone curves. 
Through the identification $S^1 \simeq \bR/2\pi \bZ$, we regard a Jordan curve $c \colon S^1 \to \bR^2$ as a $2\pi$-periodic map $c \colon \bR \to \bR^2$. 

\begin{definition}[{\cite[\S6]{Stromquist}}]\label{def:locally_monotone}
    A Jordan curve $c\colon \bR \to \bR^2$, regarded as a $2\pi$-periodic map, is said to be \emph{locally monotone} if for any $p\in \bR$, there exist an open connected neighborhood $U_p \subset \bR$ of $p$ and a unit vector $\vec{v}(p)$ such that the inner product $q \mapsto c(q) \cdot \vec{v}(p)$ is a strictly monotone function on $U_p$.
\end{definition}

\begin{proposition}
    A locally monotone Jordan curve $\cur$ admits a continuous Legendrian lift. 
\end{proposition}
\begin{proof}
    Let $p \in \bR$ and define $g_p(q) \coloneqq c(q) \cdot \vec{v}(p)$, a strictly monotone function on $U_p$.
    We define a function $f_p$ on $U_p$ as follows:
    \[
        f_p(q) \coloneqq \int_{g_p(p)}^{g_p(q)} c(g_p^{-1}(q')) \cdot \vec{n}(p) \, dq' + h_p(c(q)) \quad (q \in U_p),
    \]
    where 
    \begin{itemize}
        \item $\vec{n}(p)$ is a unit vector orthogonal to $\vec{v}(p)$ such that $(\vec{v}(p), \vec{n}(p))$ forms an oriented basis of $\bR^2$;
        \item $(x_p, \xi_p)$ is the coordinate function with respect to the orthonormal basis $(\vec{v}(p),\vec{n}(p))$; and
        \item $h_p \colon \bR^2 \to \bR$ is a smooth primitive function of $\xi dx-\xi_pdx_p$.
    \end{itemize}
    After choosing appropriate constant shifts, we can glue the family of local functions $(f_p \colon U_p \to \bR)_{p \in \bR}$ to get a continuous function $f$ on $\bR$. 
    Note that a smooth Jordan curve $c$ is locally monotone, and in this case $f$ constructed above is a primitive function of $c^*\lambda=c^*(\xi dx)$. 

    We fix a non-negative smooth function $\chi\in C^\infty (\bR)$ supported on $[-1,1]$ such that $\int_\bR \chi(q) \, dq=1$.  
    For $n \in \bZ_{\ge 1}$, we take $\delta_n>0$ such that $|p-p'|<\delta_n$ implies $\|c(p)-c(p')\|<1/n$ and define 
    \[
        c_n(p) \coloneqq \int_{\bR} \delta_n^{-1}\chi (\delta_n^{-1}u) \, c(p-u)du
    \]
    for $p \in \bR$. 
    Then $c_n$ satisfies $\|c(p)-c_n(p)\| <1/n$ for any $p \in \bR$ and is a smooth Jordan curve for a sufficiently large $n$. 
    In particular, the sequence $(c_n)_n$ converges to $c$ in the $C^0$-sense.
    
    We can check from argument in Stromquist~\cite{Stromquist} that the sequence of primitives for $c_n$'s converges to $f$. 
    Indeed, by shrinking $U_p$ if necessary, $g_{n,p}(q)\coloneqq c_n(q)\cdot \vec{v}(p)$ is strictly monotone on $U_p$ and the functions $c_n(g_{n,p}^{-1}(\mathchar`-)) \cdot \vec{n}(p)$ defined on a neighborhood of $g_p(p)$ converge to $c(g_p^{-1}(\mathchar`-))\cdot \vec{n}(p)$ in the $C^0$-sense. 
\end{proof}

\begin{corollary}\label{cor:locallymonotone_peg}
    Every locally monotone Jordan curve inscribes a $\theta$-rectangle for any $\theta \in (0,\pi)$.
\end{corollary}

\section*{Acknowledgments}

During the preparation of this paper, we learned that St\'{e}phane Guillermou had proved a result similar to \cref{thm:main_intro}.
We are grateful to him for generously sharing his insights with us. 
His ideas have clarified our discussions and improved our results.
We thank Tatsuya Miura for letting us know the Square Peg Problem many years ago and the helpful discussions about Jordan curves. 
We also thank Vincent Humili\`ere for discussions about Jordan curves, Kaoru Ono for helpful comments related to \cref{rem:unobstructed}, and Takuya Murayama for some references about conformal mappings. 
We are grateful to Joshua Evan Greene and Andrew Lobb for pointing out an error in the earlier version.

\printbibliography

\end{document}